\documentclass[11pt]{amsart}

\usepackage[marginratio=1:1,height=600pt,width=374pt,tmargin=110pt]{geometry}

\usepackage{amsthm}
\usepackage{breqn}
\usepackage{amsmath}
\usepackage{amsfonts}
\usepackage{amssymb}
\usepackage{tikz}
\usepackage{hyperref}
\usepackage{microtype} 
\usepackage{mathtools}
\usepackage{graphicx}
\usepackage{stmaryrd}

\newcommand{\R}{\mathbb{R}} 
\newcommand{\e}{\varepsilon}
\newcommand{\N}{\mathbb N}
\newcommand{\Z}{\mathbb Z}
\newcommand{\p}{\mathbb P}
\newcommand{\Q}{\mathbb Q}

\newcommand{\M}{{\bf M}}

\newcommand{\ind}{\mathrm{Ind}}
\newcommand{\hm}{\mathcal{H}}

\newcommand{\znm}{\mathcal{Z}_{n-1}(M,\Z/2)}

\newtheorem{theorem}{Theorem}[section]
\newtheorem*{theorem*}{Theorem}
\newtheorem{lemma}[theorem]{Lemma}
\newtheorem{corollary}[theorem]{Corollary}
\newtheorem{proposition}[theorem]{Proposition}

\newtheorem{claim}{Claim}

\theoremstyle{definition}

\newtheorem{definition}[theorem]{Definition}
\newtheorem*{remark}{Remark}
\newtheorem{remarq}[theorem]{Remark}

\title{The Allen-Cahn equation on closed manifolds}
\author{Pedro Gaspar and Marco A. M. Guaraco}
\address{Instituto de Matem\'atica Pura e Aplicada (IMPA) \\ Estrada Dona Castorina 110 \\ 22460-320 Rio de Janeiro \\ Brazil}
\email{marcomen@impa.br,phgms@impa.br}

\thanks{Both authors were partly supported by CNPq-Brazil and NSF-DMS-1104592.}

\begin{document}
 
 
\begin{abstract} We study global variational properties of the space of solutions to $-\e^2\Delta u + W'(u)=0$ on any closed Riemannian manifold $M$. Our techniques are inspired by recent advances in the variational theory of minimal hypersurfaces and extend a well-known analogy with the theory of phase transitions. First, we show that solutions at the lowest positive energy level are either stable or obtained by min-max and have index 1. We show that if $\e$ is not small enough, in terms of the Cheeger constant of $M$, then there are no interesting solutions. However, we show that the number of min-max solutions to the equation above goes to infinity as $\e\to 0$ and their energies have sublinear growth. This result is sharp in the sense that for generic metrics the number of solutions is finite, for fixed $\e$, as shown recently by G. Smith. We also show that the energy of the min-max solutions accumulate, as $\e\to 0$, around limit-interfaces which are smooth embedded minimal hypersurfaces whose area with multiplicity grows sublinearly. For generic metrics with ${\rm Ric}_M>0$, the limit-interface of the solutions at the lowest positive energy level is an embedded minimal hypersurface of least area in the sense of Mazet-Rosenberg. Finally, we prove that the min-max energy values are bounded from below by the widths of the area functional as defined by Marques-Neves.
\end{abstract}

\maketitle

\setcounter{tocdepth}{1}

\section{Introduction}
Applications of variational methods to the theory of semilinear elliptic PDE is by now a vast and well-developed subject. For a large class of nonlinearities, available techniques provide information such as the number of solutions and some of their properties.

In this article, we discuss the case of nonlinearities given by derivatives of double-well potentials on closed manifolds. More precisely, we are interested in solutions to the elliptic Allen-Cahn equation on a closed manifold $M$, i.e. $u:M \to \R$ such that \begin{equation}\label{allencahn}-\e \Delta u + W'(u)/\e =0,\end{equation} where the nonlinearity is assumed to be the derivative of a double-well potential $W$ (e.g. $W(u)=(1-u^2)^2/4$). This equation arises in the study of phase transition interfaces of metal alloys \cite{AllenCahn}. 

The elliptic Allen-Cahn equation and its parabolic counterpart have been extensibly studied in the last decades for its connections with the theory of minimal hypersurfaces on $\R^n$, with an important source of motivation being De Giorgi's conjecture (see \cite{Savin} and the references therein). 

For compact domains, research have been driven by understanding the limit behavior of the solutions as $\e\to 0$. Since the late 70's it was expected that the level sets of the solutions would resemble minimal hypersurfaces as $\e\to 0$. This is indeed the case in many different situations some of which we briefly discuss in the next subsection of this Introduction. However, there are still many natural open questions concerning the  properties of the solutions such as its multiplicity, behavior of its energy values, Morse index, size of the nodal sets, number of nodal domains, etc. 

In the first part of this work we describe solutions at the lowest positive energy level. Later we show that, although equation (\ref{allencahn}) does not have infinite solutions in general \cite{GSmith}, the number of solutions grows to infinite as $\e\to 0$. In the final sections we describe the behavior of the energies of such solutions. More precise statements of our results can be found later on this Introduction. 

For some other semilinear elliptic equations results like these have been proven. This is a vast subject, we refer the reader to the surveys by Ambrosetti \cite{AmbrosettiSurvey}, Ekeland-Ghoussoub \cite{EkelandGhoussoubSurvey}, Rabinowitz \cite{RabinowitzSurvey} and the references therein. However, we should not expect some of these methods to work for equations with potentials of the type we consider in this work. We discuss this on Remark 2 at the end of this Introduction.

 In this respect and in a general sense, it is on the intent of this work to show with concrete examples that the analogy with the theory of minimal hypersurfaces can help to answer  some of these questions as well as provide directions for what the answers should be in other cases. 

\subsection*{The analogy with minimal hypersurfaces} Connections between the theories of phase transitions and minimal hypersurfaces have brought the attention of many mathematicians since the works of Modica \cite{Modica} and Sternberg \cite{Sternberg} in the 80's. 

The analogy we have in mind begins, informally, with the observation that solutions to (\ref{allencahn}) have the remarkable property that its level sets $u^{-1}(s)$ accumulate, as $\e\to 0$, around a minimal hypersurface on $M$, i.e. a critical point of the area functional. Several formal instances of this statement exist in the literature depending on the variational characteristics of the solutions, e.g. see \cite{Modica,Sternberg,TonegawaPadilla,HutchinsonTonegawa,TonegawaWickramasekera}. 

To state a particular case which is in the interest of the present work, remember that solutions to the equation above are variational objects in the sense that they are critical points of the energy functional $$E_\e(u) = \int_M \frac{\e}{2}|\nabla u|^2 + \frac{W(u)}{\e}, \ \ u \in H^1(M),$$ i.e. $E'_\e(u)\equiv0$. In this way, properties such as stability or finite Morse index of a solution $u$ are defined as usual, i.e. with respect to the bilinear form corresponding to $E_\e '' (u)(\cdot,\cdot)$, the second derivative of the energy in $H^1(M)$.

In \cite{Guaraco}, based on the recent works \cite{HutchinsonTonegawa,TonegawaWickramasekera}, the following theorem was proved.
\subsection*{Theorem A}{\it
Let $M$ be a $n$-dimensional closed Riemannian manifold and $u_{k}$ a sequence of solutions to \text{(\ref{allencahn})} in $M$, with $\e=\e_k\to0$. Assume that their Morse indices, $\sup_M |u_k|$ and $E_{\e_k}(u_k)$ are bounded sequences. Then, as $\e_k\to 0$, its level sets accumulate around a minimal hypersurface, i.e. a critical point of the area functional, smooth and embedded outside a set of Hausdorff dimension at most $n-8$.}

\

In the same work this theorem was applied to a sequence of solutions obtained by a single parameter min-max construction. Such solutions have Morse index less than or equal to 1. As a corollary one concludes the celebrated Almgren-Pitts existence theorem for minimal hypersurfaces. 

This approach to Almgren-Pitts theorem can be interpreted as a converse of the results of Pacard and Ritor\'e \cite{PacardRitore}. They proved that given a non-degenerated oriented separating minimal hypersurface on $M$, there exists a sequence of solutions to the Allen-Cahn equations, with $\e\to0$, whose energy accumulates around this minimal hypersurface.

For more references on this analogy we refer the reader to the works \cite{Ilmanen,MizunoTonegawa,PisanteAllen13,Pacard,TonegawaSurvey}.

\subsection*{List of results and organization} In Section \ref{sec:least} we discuss properties of solutions to (\ref{allencahn}) with \textit{least positive energy}. These are solutions with energy equal to $$i_\e=\inf \{ E_\e(u) : u \in H^1(M),\  E'_\e(u)\equiv 0 \text{ and } E_\e(u)>0 \}.$$ This section is motivated by the work of Mazet-Rosenberg on the minimal hypersurfaces of least area \cite{Rosenberg}. We show (see Theorem \ref{least_energy}) \subsection*{Theorem 1}\textit{ $i_\e$ is attained by solutions which are either stable or obtained by min-max and have index 1. Moreover, $\infty> \liminf_{\e\to 0} i_\e > 0$ and, by Theorem A, the limit-interface of a sequence of such solutions is a smooth embedded minimal hypersurface. For generic metrics on $M$ with ${\rm Ric}_M>0$ this limit-interface is a minimal hypersurface of least area as defined by Mazet-Rosenberg. }

\

We end Section \ref{sec:least} by improving a result proved in \cite{GSmith} by a bifurcation analysis (see Proposition \ref{nosolutions}). 
\subsection*{Proposition 2}{\it If $\e>0$ is big enough, then the only solutions to (\ref{allencahn}) are the constants where the potential is critical.}

\

Our proof is different and uses only geometric inequalities. The result is improved in the sense that we provide an estimate for how big $\e$ must be in terms of the Cheeger constant of $M$.

Concerning the number of solutions to equation (\ref{allencahn}), G. Smith \cite{GSmith} recently proved that for generic metrics on a closed manifold $M$ there are only finitely many solutions for a fixed $\e>0$. On the other hand, even if an infinite number of solutions cannot be expected in general, standard min-max methods can be applied to show that the number of solutions always grows to infinity as $\e\to 0$. 

More precisely, Section \ref{sec:min-max} contains a cohomological min-max construction of solutions to equation  (\ref{allencahn}). Using the Fadell-Rabinowitz cohomological index ${\rm Ind}_{\Z/2}$ we define a sequence of families $\mathcal{F}_p$ of compact subsets of $H^1(M)$, for $p\in \N$, in the following way. An element $A\in\mathcal{F}_p$ is defined as a symmetric compact subset with a topological complexity given by ${\rm Ind}_{\Z/2}(A)\geq p+1$. Since the families $\mathcal{F}_p$ are decreasing on $p$, the min-max energy values $$c_\e(p)= \inf_{A \in {\mathcal{F}}_p} \sup_{x \in A}E_\e(x)$$ form an increasing sequence. 

Denote by $K_a$ the critical points of $E_\e$ with energy equal to $a$. An application of standard min-max methods yields (see Theorem \ref{thm:minmax})

\subsection*{Theorem 3.} \textit{ Fix $\e>0$. Then
\begin{enumerate}
\item For every $p \in \N$, it holds $c_\e(p) \leq E_\e(0) = \frac{{\rm Vol}(M)}{\e} W(0)$.
\item If $c_\e(p)< E_\e(0)$, then there exists a solution $u \in K_{c_\e(p)}$ with $|u|\leq 1$ and Morse index less or equal than $p$. Moreover, if $c_\e(p)=c_\e(p+k)<E_\e(0)$ for some $k \in \N$, then there are infinitely many solutions with the same energy and index bound.
\item $c_\e(p) = E_\e(0)$, for $p$ large enough (depending on $\e>0$ and $M$).
\end{enumerate}
} 

\

This result does not rule out the possibility that $c_\e(p)=E_\e(0)$ for a fixed $p$ and every $\e$ small. So it still does not follows that the number of solutions must grow as $\e\to 0$. However, since $E_\e(0) = {\rm Vol}(M)W(0)/\e \to \infty$ it is enough to show that for every $p$ it holds $\limsup_{\e\to 0} c_\e(p) < \infty$.

Sections \ref{sec:upperbound} and \ref{sec:bounds} contain the proofs of the following upper and lower sublinear bounds for the energy values of these min-max solutions, respectively.

\subsection*{Theorem 4.}\textit{There exists a constant $C=C(M,W)>1$ such that the min-max values $c_\e(p)$ satisfy the following sublinear bounds $$C^{-1}p^{\frac{1}{n}} \leq \liminf_{\e \to 0} c_\e(p) \leq \limsup_{\e \to 0} c_\e(p) \leq C p^{\frac{1}{n}}$$ for all $p \in \N$.}    

\

The proofs of these sublinear bounds are inspired by similar computations by Gromov \cite{Gromov}, Guth \cite{Guth} and Marques-Neves \cite{MarquesNevesInfinite} for sweepouts of the area functional. In this sense, our work brings the analogy between phase transitions and minimal hypersurfaces to a high parameter global variational context.

Combining Theorem 3 with the upper bound from Theorem 4, we obtain

\subsection*{Corollary 5.} \textit{On any closed Riemanian manifold the number of solutions to the elliptic Allen-Cahn equation (\ref{allencahn}) goes to infinity as $\e\to 0$. Moreover, for every $p$, Theorem A implies that the min-max solutions at the level $c_\e(p)$ have a limit-interface that is a smooth embedded minimal hypersurface}.

\

This article ends in Section \ref{sec:comp}, where we compare the sequence of values $\liminf_{\e \to 0} c_\e(p)$ with the widths $\omega_p(M)$ of the area functional on $M$, as defined by Marques-Neves \cite{MarquesNevesInfinite}. These widths form a sequence of real values which Gromov \cite{Gromov} proposed to consider as a nonlinear spectrum of $M$, by analogy with the min-max construction of the eigenvalues of the Laplacian. In the same way, the sequence $l_p(M)=2\sigma^{-1} \liminf_{\e \to 0} c_\e(p)$ can be considered as phase transition nonlinear spectrum of $M$, where $\sigma$ is an energy renormalization constant that depends only on $W$ (see Notation below). We show
 
\subsection*{Theorem 6.} $$\omega_p(M) \leq l_p(M) = 2\sigma^{-1}\liminf_{\e \to 0} c_\e(p),$$ for every $p \in \N$.

\

In \cite{MarquesNevesIndex}, Marques and Neves showed that $\omega_p(M)$ is achieved by the area, with multiplicity, of a minimal hypersurface $V_p \subset M$ with Morse index $\leq p$. Theorem 6 implies that the area of $V_p$ is not larger than the area of the minimal hypersurface obtained by applying Theorem A to the solutions given by Theorem 3, whose area with multiplicity is $l_p(M)$.

\subsection*{Remark 1.} In light of the analogy with minimal hypersurfaces, we have chosen to work with cohomological families. The sets $A \in \mathcal{F}_p$ may be seen as the analogue of Gromov-Guth and Marques-Neves \cite{Gromov,Guth,MarquesNevesInfinite} high-parameter cohomological sweepouts in the context of phase transitions.  However, there are other natural candidates for the families $\mathcal{F}_p$ of the min-max construction. 

The topological complexity of these families may be described in terms of other topological invariants such as homotopy, homology, and Lusterik-Schnirelmann category. This approach is reminiscent of Lusternik-Schnirelmann theory of critical points, see \cite{LusternikSchnirelmann}. Some examples of such families are considered in the works of Krasnoselskii \cite{Krasnoselskii}, Ambrosetti-Rabinowitz \cite{AmbrosettiRabinowitz}, Bahri-Berestycki \cite{BahriBerestycki} and Bahri-Lions \cite{BahriLions}, and a general account of a min-max theory in this setting is developed by N. Ghoussoub in \cite{Ghoussoub-A,Ghoussoub}. The techniques we employ can be applied to these other families to obtain existence results and sublinear bounds for the min-max values of these other for families.

\subsection*{Remark 2.} For certain potentials the Morse index of a solution gives a lower bound for its energy. This observation has been used, e.g. Bahri-Lions \cite{BahriLions}, to prove sublinear bounds for the energy of min-max solutions in some cases. We should not expect such a thing to hold for solutions to equation (\ref{allencahn}) in view of the analogy with the theory of minimal hypersurfaces. In fact, remember that in this analogy the area of a surface corresponds to the energy of a solution and there are examples of minimal surfaces showing that the index and the area are not related. It was shown in \cite{ColdingMinicozziStable,Dean,Kramer} that there is an open set of metrics on $\mathbb{S}^3$, for which there exists stable surfaces with area going to infinity. Conversely, there also exist examples of sequences of surfaces for which the Morse Index is unbounded while the area remains bounded (see \cite{Carlotto,Ambrozio}). It makes sense to expect the existence of similar examples for solutions to (\ref{allencahn}). 

\subsection*{Notation.} \label{notation} Along this work we will use the following notation \hfill

\medskip
\begin{tabular}{lll}
$W$ && a double-well potential (see \ref{sec:least})\\
$\sigma$ && the energy constant $\sigma=\int_{-1}^{1}\sqrt{W(s)/2}\,ds$\\  
$\mathrm{Inj(M)}$ && the injectivity radius of $M$\\
$H^1(M)$ && the Sobolev space of $L^2(M)$ functions with weak \\ &&derivatives also in $L^2(M)$\\
$\mathcal{H}^\rho$ && the Riemannian $\rho$-dimensional Hausdorff measure in $M$\\
$H^j(X;\Z/2)$ && the $p$-th Alexander-Spanier cohomology group of $X$ with\\ &&coefficients in $\Z/2$\\
$\mathrm{Ind_{\Z/2}}$ && the Fadell-Rabinowitz cohomological $\Z/2$-index (see \ref{sec:min-max})\\
$\mathrm{dist}_K(\cdot,A)$ && the distance function from a closed set $A \subset K$\\
${\bf I}_k(M,\Z/2)$ && the space of $k$-dimensional integral currents modulo $2$ in $M$\\ &&(see \cite[4.2.26]{Federer})\\
$\mathcal{Z}_k(M,\Z/2)$ && the space of integral currents $T \in {\bf I}_k(M,\Z/2)$ with $\partial T=0$\\
${\bf M}$ && the mass of a current $T \in {\bf I}_k(M,\Z/2)$\\
$\mathcal{F}(T)$ && the flat norm of a current $T \in {\bf I}_k(M,\Z/2)$\\
$B_r^{\mathcal{F}}(T)$ && the ball centered at $T$ with radius $r$ in $\mathcal{Z}_k(M,\Z/2)$, in the\\ &&flat norm\\
$\llbracket U \rrbracket$ && the integral $n$-dimensional mod $2$ current associated with\\ &&an open set of finite perimeter $U \subset M$.
\end{tabular}
\medskip

We will use also the following notation concerning cubical complexes. Given $m \in \N$, we will write $Q^m=[-1,1]^m$. We regard this space as a cubical complex whose cells are given by $\alpha = \alpha_1 \otimes \ldots \otimes \alpha_m$ where $\alpha_i$ is either $[-1]$,$[0$], $[1]$, $[-1,0]$, or $[0,1]$. Following \cite{pitts} and \cite{MarquesNevesWillmore}, we consider the $1$-dimensional cubical complex $Q(1,k)$ on $Q^1$ whose 0-cells and 1-cells are
	\[\left\{[i\cdot 3^{-k}]\right\}_{i=-3^k}^{3^k} \quad \mbox{and} \quad \left\{[i3^{-k}, (i+1)3^{-k}]\right\}_{i=-3^k}^{3^k-1}, \]
respectively, and the $m$-dimensional cubical complex on $Q^{m}$
	\[Q(m,k) = Q(1,k) \otimes \ldots \otimes Q(1,k) \quad \mbox{($m$ times)}\]
Given a subcomplex $X$ of some $Q(m,k)$, we denote by $X(j)$ the cubical subcomplex of $Q(m,k+j)$ obtained by further subdividing the $1$-cells of $X$ in $3^j$ intervals, that is, $X(j)$ is the union of all cells of $Q(m,k+j)$ whose support is contained in some cell of $X$. Given $q \in \N$, $X(j)_q$ will denote the set of $q$-cells of $X(j)$. We say that two vertices $x,y \in X(j)_0$ are \emph{adjacent} if they are vertices of a common $1$-cell in $X(j)_1$. Similarly, we denote by $I(m,k)$ and $I_0(m,k)$ the usual cubical complexes in $I^m=[0,1]^m$ and in its boundary $\partial I^m$, respectively. Moreover, given $j,j'\in \N$, we denote by ${\bf n}(j,j'):X(j)_0 \to X(j')_0$ the map defined by requiring that ${\bf n}(j,j')(x)$ is the closest vertex of $X(j')_0$ to $x \in X(j)_0$.

\subsection*{Acknowledgements.} This work is partially based on the first author Ph.D. thesis at IMPA. We are grateful to our advisor, Fernando Cod\'a Marques, for his constant encouragement and support. We are also grateful to the Mathematics Department of Princeton University for its hospitality. The first drafts of this work were written there while visiting during the academic year of 2015-16.

\section{Low energy levels of \texorpdfstring{$E_\e$}{E_e}}\label{sec:least}
In this section, we present a variational study of the lowest critical levels of 
	$E_\e(u) =  \int_M\frac{\e}{2} |\nabla u|^2 + \frac{W(u)}{\e} , u \in H^1(M),$
where $W$ is assumed to be a double-well potential as in \cite{Guaraco}:

\

\begin{itemize}
\item [A.\ ] $W\geq 0$, with exactly three critical points, two of which are non-degenerated minima at $\pm1$, with ${W(\pm1)=0}$ and $W''(\pm1)>0$, and  the third a local maximum $\gamma\in(-1,1)$.
\end{itemize}

\

More precisely, we are concerned with the existence and variational properties of least \textit{positive} energy solutions to the Allen-Cahn equation:
	\begin{equation} \label{eq:eac}
		-\e \Delta u + \frac{1}{\e}W'(u) = 0.
	\end{equation}
By restricting our description to solutions with positive energy we immediately exclude out the trivial solutions $\pm1$, whose energy is zero. Of course, for large values of $\e$, the least positive energy solution might be the constant $\gamma$, which is a critical point of $W$. In fact, as it follows from Proposition \ref{nosolutions}, this is the case whenever $\e>\e_0=2\sqrt{C}/h(M)$, where $C$ is a positive constant depending only on $W$ and $h(M)$ is the Cheeger constant of $M$. 

However, by \cite{Guaraco}, we already know that for small values of $\e$, non-constant solutions with energy less than $E_\e(\gamma)$ always exist. Because of this, we are also interested in describing the limit behavior of the least positive energy solutions as $\e\to 0$, in terms of the minimal hypersurface that arises as its limit-interface. In some cases, we are able to conclude that this limit-interface is in fact a minimal hypersurface with least energy, in the sense of Rosenberg-Mazet \cite{Rosenberg}. 
	
	\
	
The main result of this section is the following:

\begin{theorem} \label{least_energy} 

\

\begin{enumerate}
\item For every $\e>0$, there exists a solution $u_\e$ of (\ref{eq:eac}) with least positive energy, i.e. $$E_\e(u_\e)=\min\{ E_\e(u) : u \in H^1(M),\ E_\e'(u)=0 \text{ and } E_\e(u)>0 \}.$$

\item The least energy solution $u_\e$ is either stable or it is obtained by a one-parameter min-max and has Morse index 1. In that case $$E_\e(u_\e) = c_\e = \inf_{h \in \Gamma} \sup_{t\in[-1,1]} E_\e(h(t)),$$ where $\Gamma = \{h \in C([-1,1]: H^1(M)) :  h(\pm1)\equiv\pm1\}.$

\  

\item $\liminf_{\e\to 0} E_\e(u_\e)>0$. In particular, there is a rectifiable integral varifold $V$ such that \begin{enumerate}
\item [(i)] $\|V\|=\frac{1}{2\sigma}\liminf_{\e\to 0} E_\e(u_\e)$;
\item [(ii)] $V$ is stationary in $M$;
\item [(iii)] $\mathcal{H}^{n-8+\gamma}(\operatorname{sing}(V))=0,$ for every $\gamma>0$;
\item [(iv)] $\operatorname{reg}(V)$ is an embedded minimal hypersurface.
\end{enumerate}

\

\item If in addition, $3\leq n\leq 7$ and the metric on $M$ is bumpy with ${\rm Ric}_M >0$, then ${\rm supp} \| V\|$ is a smooth minimal hypersurface of least area among all minimal hypersurfaces, as studied by Mazet-Rosenberg in \cite{Rosenberg}. In particular, $V$ has multiplicity one and it is realized by a connected orientable smooth hypersurface of index one.

\end{enumerate}  
\end{theorem}

We prove each item of Theorem \ref{least_energy} separately. We will need the following two technical lemmas in which we point out other properties of $E_\e$. The proof of the first lemma can be found at the end of this section while the second lemma, concerning the semilinear heat flow for the equation, follows from the standard theory of semilinear parabolic equations \cite{HarauxCazenave}.  

Define $$e_\e(u)=\frac{\e |\nabla u|^2}{2}+\frac{W(u)}{\e}$$ and \textit{the discrepancy function} $$\xi_\e(u)=\frac{\e |\nabla u|^2}{2}-\frac{W(u)}{\e}.$$

\begin{lemma}\label{technical1}\

\begin{enumerate}
\item If $u$ is a critical point of $E_\e$ different from the constants $\pm 1$ or $\gamma$, then $|u|<1$ and the function $u-\gamma$ does not have a sign.

\item Palais-Smale condition: Let $u_k$ be a sequence of functions in $H^1(M)$ such that $|u|\leq 1$, $\sup_k E_\e(u_k) <\infty$ and $E'_\e(u_k) \to 0$. Then, there is a subsequence converging to a critical point of $E_\e$ in the $H^1(M)$-norm.

\item Bounded discrepancy: there exists $\e_1>0$ and $c_1 \in \R$ such that for all $\e \in(0,\e_1)$ we have $$\sup_M \xi_\e \leq c_1.$$

\item Monotonicity formula: there exist $\rho_1>0$ and $m>0$ such that for all $\e>0$, $0<\rho<\rho_1$ and $p\in M$ we have  $$\frac{d}{d\rho} \bigg(e^{m\rho} \rho^{-n+1} \int_{B_\rho(p)}e_\e(u) \bigg) \geq e^{m\rho} \rho^{-n+1} \int_{B_\rho(p)}(-\xi_\e)(u).$$ 

\item Lower bound for the potential: there exists $\e_2>0$ and $c_2>0$ such that for every $u\in H^1(M)$ that is a critical point of $E_\e$, i.e. $E'_\e(u)=0$, with $\e \in (0,\e_2)$ and every $p\in M$ with $u(p)=\gamma$, it holds  $$W(u)\big|_{B_{\e/2}(p)} \geq c_2 .$$

\end{enumerate}
\end{lemma}

Consider the parabolic Allen-Cahn equation given by 
\begin{equation}\label{parabolic}\partial_t u - \Delta u + W'(u)/\e^2 = 0. \end{equation}

Define $$\mathcal{S}=\{ u \in C^3(M) : |u|\leq 1 \}$$

\begin{lemma}\label{technical2} \

\begin{enumerate}
\item For each $u\in \mathcal{S}$, there is a unique solution $\Phi^t(u)$ of equation (\ref{parabolic}) which exists for all $t>0$.
\item For each $t>0$, $\Phi^t : \mathcal{S} \to \mathcal{S}$.
\item For each $u \in \mathcal{S}$, $E_\e(\Phi^t(u)) \leq E_\e(\Phi^s(u))$ if $t > s \geq 0$ and equality holds if and only if $u$ is a critical point of $E_\e$. In that case $\Phi^t(u)=u$ for all $t>0$.
\item For each $u, v \in \mathcal{S}$, such that $u<v$, it holds $\Phi^t(u)<\Phi^t(u)$ for all $t>0$.
\item For each $u\in \mathcal{S}$, there exist $ \Phi^{\infty}(u) \in \mathcal{S}$ which is critical point of $E_\e$ and a sequence $t_k \to \infty$, such that $\Phi^{t_k}(u) \to \Phi^{\infty}(u)$ in $H^1(M)$.

\end{enumerate}
\end{lemma}

\begin{proof}[Proof of Theorem \ref{least_energy} - Item (1)]  The set $$\{ u \in H^1(M) : \ E_\e'(u)=0 \text{ and } E_\e(u)>0 \}$$ is non-empty since $\gamma$ is always a critical point of $E_\e$ with positive energy. Since every $u$ in this set is bounded, i.e. $|u| < 1$ by Lemma \ref{technical1} (1), every minimizing sequence for $E_\e$ is precompact in $H^1(M)$ by the Palais-Smale condition, Lemma \ref{technical1} (2). Then, it suffices to prove that $\pm 1$ are isolated solutions of \eqref{eq:eac}. This follows from the Morse Lemma (see e.g. \cite[Section 9.2]{Ghoussoub}) provided we show that $d^2E_\e(\pm 1)$ are isomorphisms, where we denote by $d^2E_\e(u):H^1(M) \to H^1(M)$ the linear operator associated to the bilinear form $E_\e''(u)(\cdot, \cdot)$. Recall that $d^2E_\e(u)$ is given by (see \cite[Prop. 4.4]{Guaraco}):
	\[\left\langle d^2E_\e(u)v, w \right\rangle = \e \int_M \nabla v \cdot \nabla w + \frac{1}{\e} \int_M W''(u)vw.\]
Since $W''(\pm 1)>0$, by Hypothesis A, we see that $\left\langle d^2E_\e(\pm 1)v, v \right\rangle$ is always nonzero when $v\neq 0$. Hence, $d^2E_\e(\pm 1)$ is an isomorphism.

\end{proof}

\begin{proof}[Proof of Theorem \ref{least_energy} - Item (2)] Assume that $u_\e$ is not stable. We will construct a sweepout $h_\e \in \Gamma$ having $u_\e = h_\e(0)$ as a unique point of maximum for the energy. If $u_\e=\gamma$, then we can join $\pm 1 \in H^1(M)$ linearly to construct such a sweepout. Thus, we may assume $u_\e\neq \gamma$. By Lemma \ref{technical1} - Item (1), there are $x^{\pm} \in M$ such that $u_\e(x^-) < \gamma < u_\e(x^+)$. Let $v \in H^1(M)$ be a positive eigenfunction associated to the first eigenvalue of the stability operator of $E_\e''(u_\e)$. There exists $t_0\in(0,1)$ such that
	\[E_\e(u \pm tv) < E_\e(u_\e), \quad \mbox{whenever} \quad 0<t\leq t_0.\]
We will now describe how to join $u_\e\pm t_0v$ and $\pm 1$ using Lemma \ref{technical2}. First, we observe that $u \in \mathcal{S}$, by Lemma \ref{technical1} (1). From Lemma \ref{technical2} (5), we can find a sequence $\{t_i\}$ of positive real numbers such that $\Phi^{t_i}(u_\e+ t_0v)$ converges in $H^1(M)$ to a solution of \eqref{eq:eac} when $i \to \infty$. Moreover, the energy of $\Phi^t(u_\e + t_0v)$ is decreasing in $t$. Hence, $\Phi^{t_i}(u_\e + t_0 v)$ converges either to $+1$ or $-1$ in $H^1(M)$. We claim that
	$\Phi^{t_i}(u_\e + t_0v) \to 1, \mbox{as} \quad i \to \infty.$
In fact, notice that $u<u+t_0v$ and, from Lemma \ref{technical2} (4),
	$u = \Phi^t(u) < \Phi^t(u+t_0v),  \mbox{for all} \quad t \geq 0.$
In particular, there exists a neighborhood $V \subset M$ of $x^+$ such that $\Phi^t(u+t_0v)>\gamma$ in $V$ for all $t\geq 0$. Therefore, $\Phi^{t_i}(u+t_0v) \to 1$. Since the functions $\pm 1$ are isolated points of minimum for $E_\e$, we can find $\delta>0$ such that
	\[E_\e(u_\e)>E_\e(w), \quad \mbox{for all} \quad w \in B_\delta(1) \cup B_\delta(-1) \subset H^1(M).\]
Choose $i \in \N$ such that $||\Phi^{t_i}(u + t_0 v) - 1||<\delta$ and join $\Phi^{t_i}(u+ t_0 v)$ to $1$ by a segment contained in $B_\delta(1)$. It follows that the energy remains strictly below $E_\e(u_\e)$ along this path. We can now concatenate the paths ${t \in [0,t_0] \mapsto u+tv}$, ${t \in [0,t_i] \mapsto \Phi^t(u\pm t_0 v)}$ and this segment to obtain a continuous path $h_\e:[0,1] \to H^1(M)$ with $h_\e(0)=u_\e$ and $E_\e(h_\e(t))<E_\e(u_\e)$ for $t>0$. A similar construction shows that we can define $h_\e:[-1,0] \to H^1(M)$ so that $h_\e(-1)=-1$, $h_\e(0)=u_\e$ and $E_\e(h_\e(t))<E_\e(t)$ for $t<0$. This completes the construction of the claimed sweepout. 

We will now show that $E_\e(u_\e)=c_\e$ and that $u_\e$ has Morse index $m(u_\e)=1$. Clearly, $c_\e \leq E_\e(u_\e)$. It follows from \cite{Guaraco} that $c_\e>0$ and that there exists $v_\e \in K_{c_\e}$, which $E_\e(u_\e) \leq E_\e(v_\e) = c_\e$, since $u_\e$ is the solutions of least positive energy. Therefore $c_\e = E_\e(u_\e)$ and $u_\e$ is a min-max solution. To see that $u_\e$ has Morse index $m(u_\e)=1$, we proceed as in \cite[Prop. 3.1]{MarquesNevesDuke}. By definition, $u_\e$ has Morse index $m(u_\e)\geq 1$. If this inequality is strict, then we would be able to find linearly independent eigenfunctions $v=v_1,v_2 \in H^1(M)$ associated to negative eigenvalues of $d^2E_\e(u_\e)$ such that
	\[E_\e''(u_\e)(v_1,v_1)<0, \quad E_\e''(u_\e)(v_2,v_2)< 0, \quad \mbox{and} \quad E_\e''(u_\e)(v_1,v_2) = 0.\]
Choose also a function $\rho \in C^\infty(\R)$ such that $\rho(t)=1$ for $|t|\leq 1/3$, and $\rho(t)=0$, for $|t|\leq 1/2$. Define $\gamma_\e:\R \times [-1,1] \to H^1(M)$ by
	\[\gamma_\e(s,t):=h_\e(t) + s\rho(t)v_3, \quad (s,t) \in \R \times [-1,1].\]
We have $\gamma_\e(s, \cdot) \in \Gamma$, $\gamma_\e(0,0)=h_\e(0)=u_\e$ and:
	\[\frac{\partial}{\partial t} \gamma_\e(0,0) = h_\e'(0) = v_1, \quad \frac{\partial}{\partial s} \gamma_\e(0,0)=v_3.\]
Hence,
	\[\frac{\partial^2}{\partial t^2}(E_\e \circ \gamma_\e)(0,0) <0, \quad \frac{\partial^2}{\partial s \partial t}(E_\e \circ \gamma_\e)(0,0) =0, \quad \frac{\partial^2}{\partial s^2}(E_\e \circ \gamma_\e)(0,0) <0.\]
Since $E_\e \circ h_\e$ has an unique maximum point in $t=0$, we can find $\delta_1>0$ such that
	\[E_\e(\gamma_\e(\delta_1,t)) < E_\e(\gamma_\e(0,0)) = E_\e(u_\e)\]
for all $t \in [-1,1]$. But $\gamma_\e(\delta_1, \cdot) \in \Gamma$ and $E_\e(u_e)=c_\e$, which leads to a contradiction. Therefore, $u_\e$ has index 1. The last claim follows from \cite[Theorem 3.7]{Guaraco}.

\end{proof}

\begin{proof}[Proof of Theorem \ref{least_energy} - Item (3)]
We will use the Monotonicity formula from Lemma \ref{technical1} (4), to show that $\liminf E_\e(u_\e)>0$. 

Choose $0<\tilde\rho<\min\{ {\rm Inj}(M),\rho_1\},$ where $\rho_1$ is as in Lemma \ref{technical1}, and such that $\omega_n \rho^n/2 \leq {\rm Vol}(B_\rho(q)) \leq 2 \omega_n \rho^n$ for every $q\in M$ and every $\rho \in (0,\tilde\rho)$. This is possible given that $M$ is compact. Fix $\e>0$ so that $\e<\min\{\e_1,\e_2,\tilde\rho\}$, with $\e_1$ and $\e_2$ as in Lemma \ref{technical1}.

Choose $p\in M$ such that $u_\e(p)=\gamma$, by Lemma \ref{technical1} (1). By Lemma \ref{technical1} (4), we have
\begin{dmath*}e^{m\tilde\rho} (\tilde\rho)^{-n+1} \int_{B_{\tilde\rho}(p)}e_\e(u) \geq e^{m\e/2} (\e/2)^{-n+1} \int_{B_{\e/2}(p)}e_\e(u) + \int_{\e/2}^{\tilde\rho} e^{m\rho} \rho^{-n+1}\int_{B_\rho(p)}\left( -\xi_\e \right)(u)
\end{dmath*}

Notice that, by the choice of $\e$, and by items (3) and (5) of Lemma \ref{technical1}, we have $-\xi_\e(u) \geq -c_1$ everywhere on $M$ and $e_\e(u)\geq W(u)/\e \geq c_2/\e$ on $B_{\e/2}(p)$. Then, complementing with the inequality above it follows

\begin{dmath*}
e^{m\tilde\rho} (\tilde\rho)^{-n+1} \int_{B_{\tilde\rho}(p)}e_\e(u) \geq \frac{e^{m\e/2} c_2 \omega_n}{4} - 2c_1 \tilde\rho \omega_n e^{m\tilde\rho}
\end{dmath*}
 
We can choose $\tilde\rho>0$ small enough (and independently of $\e$) so that $e^{m \tilde\rho} \leq 2$, then 
 
$$E_\e(u_\e) \geq \int_{B_{\tilde\rho}(p)}e_\e(u) \geq \bigg(\frac{c_2}{8} - 2c_1 \tilde\rho\bigg) \omega_n \tilde\rho^{n-1}.$$

To finish the argument, we can choose $\tilde \rho$ small enough if necessary, and independently of $\e>0$, in such a way that $\frac{c_2}{16} - 2c_1 \tilde\rho$ is positive. In particular, there is a positive constant $C>0$ such that $E_\e(u_\e)\geq C$ for every $\e>0$ small enough.

The rest of the conclusions follow from the fact that the index of $u_\e$ is bounded by 1 (by Item (2)) and  Theorem A from \cite{Guaraco}.
\end{proof}

\begin{proof}[Proof of Theorem \ref{least_energy} - Item (4)]  Let $\tilde \Sigma \subset M$ be a minimal hypersurface of least area. It follows from Propositions 10 and 13 in \cite{Rosenberg}, combined with the recent catenoid estimate \cite{KetoverMarquesNeves} to rule out the case unoriented surfaces, that $\tilde \Sigma$ is separating and orientable. Theorem 1.1 in \cite{Pacard} allows one to obtain $\e_0>0$ such that, for every $\e \in (0,\e_0)$, there exists $v_\e \in K$ such that $E_\e(v_\e) \to \frac{1}{2\sigma}|\tilde \Sigma|$ as $\e\to 0$. Since $E_\e(v_\e) \geq E_\e(u_\e)$, we have $|\tilde \Sigma| \geq |\Sigma|$. Using that $\tilde \Sigma$ has least area, we conclude that $|\Sigma|=|\tilde \Sigma|$. \ \
\end{proof}

The following proposition shows that, provided $\e>0$ is big enough, the only solutions to the Allen-Cahn equation are constants. Another way of saying this is that if a domain is not big enough in terms of the potential, then there are not interesting solutions, an idea also found in \cite{BrezisOswald}. This result has been previously stated in \cite{GSmith} for a more general class of potentials. However, our proof still works in those cases, and is fairly different as we use only geometric inequalities. In addition, it provides an estimate on how big $\e$ must be in terms of the Cheeger constant of the manifold. 

\begin{proposition}\label{nosolutions}
The constants $\pm1$ and $\gamma$ are the only solutions of equation (\ref{eq:eac}), whenever $\e>\e_0=2 \sqrt{C}/h(M)$. Here $h(M)$ is the Cheeger constant of $M$ and $C=\max_{s\in[-1,1]}W'(s)/(s-\gamma)$ is a positive constant depending only on $W$.
\end{proposition}

\begin{proof}
Let $u$ be a non-constant solution of $-\e^2\Delta u + W'(u)=0$. Without loss of generality, we can assume that $\Omega=\{u>\gamma\}$ is such that ${\rm Vol}(\Omega)\leq {\rm Vol(M)}/2$, otherwise we would simply choose $\{u<\gamma\}$ and the arguments would follow in the same way. Notice that both sets are non-empty by Lemma \ref{technical1} - Item (1).

The argument is by contradiction. Our goal is to prove that if $\e>0$ is big enough, $\Omega$ cannot support $u-\gamma \in H^1_0(\Omega)$ satisfying $-\e^2\Delta u + W'(u)=0$ and $u-\gamma>0$ on the interior, as it does.

The set $\Omega$ does not necessarily have a nice boundary, so we are not able to define the first eigenvalue of the Dirichlet problem. However, we can still define the following quantity, known as its \textit{fundamental tone} (see \cite{Chavel}) $$\lambda^*(\Omega) = \inf_{v\in H_0^{1}(\Omega)} \frac{\int_\Omega |\nabla v|^2}{\int_\Omega v^2}.$$ Motivated by some computations in \cite{BrezisOswald}, take $u-\gamma$ as a test function and substitute the equation after integrating by parts (by approximation, since the domain is not necessarily  smooth, as in \cite{Chavel} pag. 21-22). We obtain $$\lambda^*(\Omega) \leq \frac{\int_\Omega |\nabla (u-\gamma)|^2}{\int_\Omega (u-\gamma)^2}= -\frac{1}{\e^2}\frac{\int_\Omega (u-\gamma)W'(u)}{\int_\Omega (u-\gamma)^2}\leq C/\e^2,$$ where $C=C(W)=\max_{s\in[-1,1]}-W'(s)/(s-\gamma)$, which is positive and finite since $W'(\gamma)=0$.

On the other hand, Cheeger's inequality gives a lower bound on $\lambda^*(\Omega)$ in terms of the Cheeger constant of $\Omega$ which is defined as $$h(\Omega)= \inf {\rm Area}(\partial A)/{\rm Vol}(A),$$ where the infimum its taken over all open sets $A\subset\subset \Omega$ with smooth boundary. 

Cheeger's inequality asserts $$\frac{h(\Omega)^2}{4}\leq \lambda^*(\Omega).$$ In fact, the proof of this inequality for smooth boundaries (as found, for example, in \cite{Chavel}, Theorem 3, pag. 95) also applies to our case. Moreover, the Cheeger's constant of $M$ is defined as $$h(M)=\inf \frac{{\rm Area}(\partial A)}{ {\rm Vol}(A)},$$ where this time the infimum ranges over all open sets $A\subset M$ with smooth boundary and such that ${\rm Vol}(A)\leq {\rm Vol}(M)/2$. This bound on the volume of $A$ is necessary whenever $M$ is a closed manifold, otherwise the constant would be trivially zero. 

Finally, $$0< \frac{h(M)^2}{4} \leq \frac{h(\Omega)^2}{4}\leq \lambda^*(\Omega) \leq \frac{C}{\e^2},$$ which yields a contradiction provided $\e>\e_0=2 \sqrt{C}/h(M)$.

\end{proof}

\subsection{Proof of the technical lemma}

\begin{proof}[Proof of Lemma \ref{technical1}]

To see (1) first notice that, by the form of the potential (i.e, Hypothesis A) if there is a point $p$ of maximum of $u$, with $u(p)>1$, then $W'(u)>0$ in a small neighborhood $U$ around $p$. In particular, $-\Delta u > 0$ on $U$. Taking a non-negative test function $\varphi$, with non-empty support on $U$ we conclude $$E_\e'(u)(\varphi) = \int_U \varphi (-\e \Delta u + W'(u)/\e) > 0,$$ which contradicts that $u$ is a critical point of $E_\e$. Thus $|u|\leq 1$. However, if $\max u =1$ then $u$ would be identically 1 by the maximum principle.

For the second part of (1), observe that if $u-\gamma$ is non-negative and not identically zero, then $$E'_\e(u)(1)=\frac{1}{\e} \int_M W'(u)<0$$ since $W'(u)<0$ whenever $\gamma<u<1,$ contradicting the fact that  $u$ is a solution. In the same way, $u-\gamma$ cannot be non-positive.

For the proof of (2) and (4) see \cite{Guaraco}. The proof of (3) is presented in \cite{HutchinsonTonegawa} for $M=\R^n$. The proof is local, and the same arguments work for general metrics.

To see (4), fix $p\in M$ and $\e>0$. Choose any $0<\delta<{\rm Inj}(M)$ and consider the function on $B_1(0) \subset T_p M$ given by $\tilde u(x) = u\ \circ \ \exp_p(\delta \cdot x)$. In these dilated normal coordinates the equation $-\Delta u+ W'(u)/\e^2=0$ has the form $$- a_{ij} \cdot \partial_{ij} \tilde u - \delta \cdot  b_i \cdot  \partial_i \tilde u + (\delta/\e)^2 \cdot  W' (\tilde u)=0,$$ where, given that $M$ is compact, the coefficients are such that there exists $C=C(M)$ satisfying $\|a_{ij}-\delta_{ij}\|_{C^1(B_1(0))} + \|\delta \cdot b_k\|_{C^1(B_1(0))}<C \delta^2$ for every $i,j,k$ and $\delta<\tilde\delta$, for some $\tilde \delta < {\rm Inj}(M)$ small enough and independent of $p$. Since we already know that $|u|\leq 1$, choosing $\delta<\tilde\delta$ so that $\delta/\e\leq1$, standard gradient estimates (see Proposition 2.19 of \cite{HanLin}) guarantee the existence of a positive constant $\tilde C=\tilde C(\tilde \delta,W)$ such that $\sup_{B_{1/2}(0)}|\nabla \tilde u | \leq \tilde C$. Whenever $\e<\tilde\delta$ we can simply choose $\delta =\e$. 

Assume now that $\tilde u(0)=u(p)=\gamma$. By the observations made above, for every $\xi>0$ there exists $\rho=\rho(\xi, \tilde \delta, W)\in(0,1/2)$ such that $|\tilde u -\gamma|<\xi$ in the ball $B_{\rho \delta}(0)\subset T_p M$. In particular, we can choose $\rho=\rho(\tilde \delta, W)>0$ so that $W(u)>\frac{1}{2}\max_{[-1,1]} W$ in $B_{\rho \delta}(p)\subset M$, independently of $p$.

\end{proof}

\section{Multiparameter min-max for the energy functional}\label{sec:min-max}

In this section, we employ min-max methods to find solutions for the equation \eqref{eq:eac} using higher dimensional families. 

\subsection*{Hypothesis on the potential $W$.} Besides from Hypothesis A from the last section, from now on we will assume that $W$ is an even function. In particular, $\gamma=0$.

\

Since $E_\e$ is an even functional, we can use families of symmetric, compact and topologically non-trivial sets in $H^1(M)$ to detect critical points of $E_\e$. These families can be seen as the analogue of the Gromov-Guth high-parameter families \cite{Guth, MarquesNevesInfinite} in the context of phase transitions. The non-triviality of these sets will be expressed in terms of a \emph{topological $\Z/2$-index}, in the sense of \cite{Ghoussoub-A} (see also \cite{Ghoussoub}). 

\subsection{A topological $\Z/2$-index} \label{z2index} By a \emph{$\Z/2$-space} we will mean a paracompact Hausdorff space $X$ with a given homeomorphism $T:X \to X$ such that $T^2 = \mathrm{id}_X$. We will say that a $\Z/2$-space is \emph{free} if $T$ has no fixed points.

\begin{definition} \label{def:indice}
Let $\mathcal{C}$ be a class of paracompact $\Z/2$-spaces and assume that $(A,T|_A) \in \mathcal{C}$ whenever $A$ is an invariant paracompact subset of $X$ and $(X,T) \in \mathcal{C}$. Assume also that $\mathcal{C}$ contains $S^\infty$. A function $\ind:\mathcal{C} \to \N \cup \{0,+\infty\}$ is called a \emph{topological $\Z/2$-index} if it satisfies the following properties:

\begin{enumerate}
	\item[(I1)](Normalization) $\ind(A) =0$ if, and only if, $A=\emptyset$.
	\item[(I2)](Monotonicity) If $A_1,A_2 \in \mathcal{C}$ and there exists an equivariant continuous map $A_1 \to A_2$, then $\ind(A_1) \leq \ind(A_2)$.
	\item[(I3)](Continuity) If $X \in \mathcal{C}$ and $A \subset X$ is an invariant closed subset of $X$, there exists an invariant neighborhood $V \subset X$ of $A$ such that $\ind(A) = \ind(\overline V)$.
	\item[(I4)](Subaditivity) If $X \in \mathcal{C}$ and $A_1,A_2 \subset X$ are invariant closed subsets, then $\ind(A_1 \cup A_2) \leq \ind(A_1) + \ind(A_2)$.
	\item[(I5)] For every compact free $\Z/2$-space $X \in \mathcal{C}$, if $\ind(X)\geq n$, then the orbit space $\tilde X$ has at least $n$ elements.
	\item[(I6)] It holds $\ind(X)<+\infty$ for all compact free $\Z/2$-space $X \in \mathcal{C}$.
\end{enumerate}
\end{definition}

In order to obtain solutions to \eqref{eq:eac} with bounded Morse index, we have chosen to work with cohomological families of subsets of $H^1(M)$ which can be described in terms of the cohomological index of E. Fadell and P. Rabinowitz \cite{Fadell-Rabinowitz1,Fadell-Rabinowitz2}.  

Given a paracompact free $\Z/2$-space $(X,T)$, one can see that there exists a continuous map $f:X \to S^\infty$ which is equivariant, that is $f(Tx) = -f(x)$ for all $x \in X$. Here, $S^\infty=\bigcup_nS^n$ is the infinite dimensional sphere, with the topology given by the direct limit of $\{S^n\}_{n \in \N}$ ordered by the inclusions $S^n \hookrightarrow S^m$, for $n\leq m$. Denote by $\tilde f: \tilde X \to \R\p^\infty$ the induced continuous map, where $\tilde X$ and $\R\p^\infty$ are the orbit spaces $X/\{x \sim Tx\}$ and $S^\infty/\{x \sim -x\}$, respectively. The Alexander-Spanier cohomology ring of the infinite dimensional projective space $\R\p^\infty$ with $\Z/2$-coefficients is isomorphic to $\Z/2[w]$, with a generator $w \in H^1(\R\p^\infty;\Z/2)$ (see \cite[p. 264]{Spanier}). The map $f$ is also unique modulo equivariant homotopy, so we define the \emph{cohomological index} of $(X,T)$ by
		\[\ind_{\Z/2}(X,T) = \sup\left\{ k : \tilde f^*(w^{k-1}) \neq 0 \in H^{k-1}(\tilde X;\Z/2)\right\}. \]
We set $w^0 = 1 \in H^0(\R\p^\infty;\Z/2)$ and adopt the convention $\ind_{\Z/2}(\emptyset,T) = 0$, so that $\ind_{\Z/2}(X,T)\geq 1$ iff $X$ is non-empty. If $(X,T)$ is a $\Z/2$-space which is not free, we set $\ind_{\Z/2}(X,T)=\infty$. We will write $\ind_{\Z/2}(X)=\ind_{\Z/2}(X,T)$ whenever the action of $\Z/2$ is clear from the context. For subsets of Banach spaces, we will assume this action is the antipodal map $x \mapsto -x$, unless otherwise stated.


In the Appendix \ref{ap:index}, we give some details about the construction of the cohomological index $\ind_{\Z/2}$, which defines a topological $\Z/2$-index in the class of all paracompact $\Z/2$ spaces.


\subsection{Setting for the multiparameter min-max} In this subsection we briefly describe the general min-max procedure we will apply to the Energy functional. 

Let $X$ be a $C^2$ Hilbert manifold which is also a free $\Z/2$-space. For each $p\in \N$, consider the family
	\begin{equation} \label{def:fp}
		\mathcal{F}_p = \left\{ A \subset H^1(M) : A \ \mbox{compact, symmetric}, \ \ind_{\Z/2}(A) \geq p+1\right\}.
	\end{equation} One easily verifies that ${\mathcal{F}}_p$ is a $p$-dimensional $\Z/2$-cohomological family, in the sense defined in \cite[\S 3]{Ghoussoub-A}.
	
Given an equivariant functional $\varphi:X \to \R$, we define the min-max values
	\[c_p=c(\varphi, \mathcal{F}_p) := \inf_{A \in {\mathcal{F}}_p} \sup_{x \in A}\varphi(x), \quad \mbox{for} \quad p \in \N.\] Since $\mathcal{F}_{p+1} \subset \mathcal{F}_p$, we have $c_p\leq c_{p+1}$ for all $p \in \N$. A sequence $\{A_n\} \subset {\mathcal{F}}_p$ is called a \emph{minimizing sequence} if \[\sup_{x \in A_n} \varphi(x) \to c_p \quad \mbox{as} \quad n \to \infty.\] Given such a sequence, we say that $\{x_n\} \subset X$ is a \emph{min-max sequence} for $\{A_n\}$ if \[d(x_n, A_n) \to 0 \quad \mbox{and} \quad \varphi(x_n) \to c_p, \quad \mbox{as} \quad n\to \infty.\]
	
In order to deal with the lack of compactness of the domain, we restrict ourselves with a certain class of functionals: we say that $\varphi$ satisfies the \emph{Palais-Smale condition along} $\{A_n\}$ if every min-max sequence $\{x_n\}$ for $\{A_n\}$ with $\varphi'(x_n) \to 0$ contains a convergent subsequence. This is the key condition that allow us to find a critical point of $\varphi$ at each min-max level $c_p$.  

For each $c \in \R$, we denote by $K_{c}$ the set of all $x \in X$ such that $\varphi'(x)=0$ and $\varphi(x)= c$. If $\varphi'(x)=0$, we write $m(x)$ and $m^*(x)$ for the Morse index and the augmented Morse index of $x$, that is, $m(x)$ and $m^*(x)$ are the maximal dimensions of subspaces of $T_xX$ such that $\varphi''(x)$ is negative definite and negative semidefinite, respectively. Given $c\in \R$ and $\ell \in \N$, we denote by $K_c(\ell)$ the set of critical points of $x \in K_c$ of $\varphi$ such that $m(x)\leq \ell \leq m^*(x)$. 

\subsection{A min-max theorem} 

In the setting above, whenever $X$ is a complete $\Z/2$-free space, the results of \cite{Ghoussoub-A} imply the existence of a critical point ${x \in K_{c_p}(p)}$, for every value of $p$. Unfortunately, the space in which we would like to apply the min-max construction is not complete. 

More precisely, let $(M^n,g)$ be a closed Riemannian manifold. Define $X=H^1(M)\backslash \{0\}$ and let $\mathcal{F}_p$ be the cohomological family defined as above. Clearly, $X$ is a $\Z/2$-free space which is not complete. However, for each $\e>0$, we still have min-max values $c_\e(p)$ defined as above, with $E_\e$ in the place of $\varphi$: \[c_\e(p)=c(E_\e, \mathcal{F}_p) := \inf_{A \in {\mathcal{F}}_p} \sup_{x \in A}E_\e(x), \quad \mbox{for} \quad p \in \N.\]  We also have that $E_\e:X \to \R$ satisfies the Palais-Smale condition along every minimizing sequence which is bounded away from 0 (see \cite[Proposition 4.4]{Guaraco}). 

In such a situation we still can apply the results from \cite{Ghoussoub-A} to minimizing sequences that are bounded away from zero, e.g. when ${c_\e(p) < E_\e(0)}= {\rm Vol}(M) W(0)/\e$, that goes to infinity as $\e \to 0$. The following theorem states that the limit behavior of the min-max values $c_\e (p)$, as $\e \to 0$, is bounded from above and from below by functions with sublinear growth on $p$. In particular, for any fixed $p$, $c_\e(p)<E_\e(0)$ will hold provided $\e$ is small enough. 

\begin{theorem} \label{thm:sublinear}
Let $M^n$ be a compact Riemannian manifold. There exists a constant $C(M)>1$ such that the min-max values $c_\e(p)$ satisfy $$C(M)^{-1}p^{\frac{1}{n}} \leq \liminf_{\e \to 0^+} c_\e(p) \leq \limsup_{\e \to 0^+} c_\e(p) \leq C(M) p^{\frac{1}{n}}$$ for all $p \in \N$.
\end{theorem}

The proof of Theorem \ref{thm:sublinear} is motivated by Guth-Gromov bend-and-cancel arguments and we postpone it to the next two sections. 

Now we can state the main result of this section.

\begin{theorem}\label{thm:minmax} 
Fix $\e>0$.

\begin{enumerate}
\item For every $p \in \N$, it holds $c_\e(p) \leq E_\e(0) = \frac{{\rm Vol}(M)}{\e} W(0)$.

\item If $c_\e(p)< E_\e(0)$, then there exists a solution $u \in K_{c_\e(p)}$ with $|u|\leq 1$ and $m(u)\leq p \leq m^*(u)$. Moreover, if $c_\e(p)=c_\e(p+k)<E_\e(0)$ for some $k \in \N$, then
	\[\ind(K_{c_\e(p)}(p+k)) \geq k+1.\]
	
\item $c_\e(p) = E_\e(0)$, for $p$ large enough (depending on $\e>0$ and $M$).

\end{enumerate}

\end{theorem}

\begin{proof}

Item (1) follows from the fact that given any $A \in \mathcal{F}_p$, its image, $\delta A$, by an homothety of factor $\delta>0$, also belongs to $\mathcal{F}_p$. In addition, $\sup_{\delta A} E_\e \to E_\e(0)$ as $\delta \to 0$, which implies that $c_\e(p)=c(E_\e, \mathcal{F}_p)\leq E_\e(0)$, by the definition of $c_\e(p).$

As we mentioned before, (2) follows from the min-max theorems for cohomological families from \cite{Ghoussoub-A}. 

To prove (3), choose $p$ such that the $(p+1)$-th eigenvalue of the Laplace operator on $M$ satisfies $\lambda_{p+1}\geq 2C/\e^{2}$, where $C=\max_{u\in[-1,1]} u^{-2}(W(0)-W(u))$. Notice that a simple application of L'Hospital's rule, when $u\to 0$, yields that $C$ is a finite positive constant. 

Let $\{f_1,\ldots, f_{p}\}$ be the corresponding first $p$ eigenfunctions. From the min-max characterization of the eigenvalues, it follows that$\int_M|\nabla u|^2/\int_M |u|^2 \geq \lambda_{p+1} \geq 2C/\e^2$, whenever $u \in H^1(M)$ is such that $\int_M uf_i =0$ for $i=1,\ldots, p$. Given $A \in \mathcal{F}_p$, we can always find $u \in A$ with this property, otherwise the map
	\[u \in A \mapsto \left( \int_M u f_1,\ldots, \int_M u f_p \right) \in \R^p\]
would induce a continuous equivariant map $A \to \R^p\setminus\{0\} \to S^{p-1}$, contradicting $A \in \mathcal{F}_p$ (see Lemma \ref{lem:lb1} below).

Observe that the truncation map $\tau:H^1(M) \to H^1(M)$ given by $\tau(u) = \max\{-1,\min\{1,u\}\}$ is continuous (by Lemma \ref{truncar}) and odd, hence $\tau(A) \in \mathcal{F}_p$ and $E_\e(u)\geq E_\e(\tau(u))$, for every $u\in H^1(M)$. Then $$\sup E_\e(A) \geq \sup E_\e(\tau(A)) \geq E_\e(u) \geq \frac{1}{\e} \int_M  Cu^2 + W(u) \geq \frac{1}{\e} \int_M W(0),$$ by the definition of $C$. Minimizing over all $A \in \mathcal{F}_p$, we obtain $c_\e(p)\geq E_\e(0)$ whenever $\lambda_{p+1} \geq 2C/\e^{2}$.

\end{proof}

Combining Theorems \ref{thm:sublinear} and \ref{thm:minmax}, we have that for every $p$ there exist a sequence $\e_k \to 0$ and sequence of $u_k \in C^3(M)$ satisfiying $- \e^2 \Delta u_k + W'(u_k)=0$, with bounded Morse index and such that $E_{\e_k}(u_k) \to \liminf_{\e\to 0} c_\e(p)=c(p)$, which is a finite positive constant for every $p$ (from Theorem \ref{thm:sublinear}). Finally, applying, Theorem A from \cite{Guaraco} we conclude.

\begin{corollary}
In every $n$-dimensional closed Riemannian manifold there exists an integral varifold $V_p$ such that
\begin{enumerate}
\item [(i)] $\|V_p\|=\frac{1}{2\sigma}c(p)$;
\item [(ii)] $V$ is stationary in $M$;
\item [(iii)] $\mathcal{H}^{n-8+\gamma}(\operatorname{sing}(V))=0,$ for every $\gamma>0$;
\item [(iv)] $\operatorname{reg}(V)$ is an embedded minimal hypersurface.
\end{enumerate}

\end{corollary}

We conclude with some remarks concerning the results from this section.

\begin{remarq} If $W(u)$ is the canonical potential $(u^2-1)^2/4$, we can show that \emph{$0 \in H^1(M)$ is the only critical point of $E_\e$ with energy $\geq E_\e(0)$}. In fact, given $u \in H^1(M)$ such that $E_\e'(u)=0$, we have $|u|< 1$ and thus
	\begin{align*}
		E_\e(u) &= \frac{\e}{2} \int_M|\nabla u|^2 + \frac{1}{\e} \int_M W(u) = -\frac{1}{2\e}\int_M W'(u)u +  \frac{1}{\e} \int_M \frac{(1-u^2)^2}{4}\\
		& = \frac{1}{\e} \int_M \left( \frac{1-2u^2+u^4}{4}- \frac{u^4-u^2}{2}  \right) = \frac{1}{4\e}\left( \mathrm{Vol}(M) - \int_M u^4 \right).
	\end{align*}
Hence, $E_\e(u) \geq E_\e(0)$ implies $u=0$. 
\end{remarq} 

\begin{remarq} We can compare the one-parameter min-max solutions given by \cite[Prop. 4.4]{Guaraco} with the one given by Theorem \ref{thm:minmax} with $p=1$, in the following way. Let
	\[\Gamma = \{ h \in C([-1,1],H^1(M)) : h(\pm 1) = \pm 1\}.\]
Given $h \in \Gamma$, we define $f_h:S^1 \to H^1(M)$ by
	\[f_h(x) = \left\{ \begin{array}{rcl} h(x_1), & \mbox{if} & x_2 \geq 0, \\ -h(-x_1), & \mbox{if} & x_2 \leq 0 \end{array} \right.,\]
for $x=(x_1,x_2) \in S^1$. Since $h(\pm 1)=\pm 1$, we see that $f_h$ is well defined and continuous. Moreover, $f_h$ is equivariant and by the monotonicity of the cohomological index it follows that
	\[A_h:=f_h(S^1)=h([-1,1]) \cup (-h([-1,1])) \in \mathcal{F}_1.\]
Hence
	\[c_\e(1) \leq \inf_{h \in \Gamma} \max_{u \in A_h} E_\e(u) = \inf_{h \in \Gamma} \max_{t \in [-1,1]} E_\e(h(t)) = c_\e\]


If the critical point of $E_\e$ with least positive energy is not stable -- e.g., if $M$ has positive Ricci curvature (see \cite{Farina}) -- then by Theorem \ref{least_energy} we get
	\[c_\e = E_\e(u_\e) \leq c_\e(1) \leq c_\e.\]
Hence, in this case, the min-max values obtained using $1$-sweepouts \cite{Guaraco} and invariant families with cohomological index $\geq 2$ coincide. 
\end{remarq}


\begin{remarq}
Some of the conclusions of the min-max Theorem \ref{thm:minmax}, as well as the upper and lower bounds for the min-max values of Sections \ref{sec:upperbound} and \ref{sec:bounds}, hold for other well known families of compact symmetric subsets of $H^1(M)$. Consider, for example, the family $\mathcal{M}_p$ given by all the images of continuous odd maps $S^p \to H^1(M)\setminus \{0\}$. A similar family was used by A. Bahri and H. Berestycki in \cite{BahriBerestycki} -- attributed to Krasnoselskii \cite{Krasnoselskii} -- to prove the existence of infinitely many solutions for a class of semilinear equations. One easily verifies that $\mathcal{M}_p$ is a \emph{homotopic} family in the sense of \cite{Ghoussoub}. Then, it is possible to obtain a corresponding min-max theorem for $\mathcal{M}_p$ that proves the existence of solutions to \eqref{eq:eac} with Morse index bounded above by $p$ and energy $c(E_\e,\mathcal{M}_p)=\inf_{A \in \mathcal{M}_p} \sup E_\e(A)$.

Another option is to replace the cohomological index $\ind_{\Z/2}$ in \eqref{def:fp} with other topological $\Z/2$-indexes. There is a natural choice which gives the largest possible families defined in these terms: it consists of all symmetric compact sets $A \subset H^1(M)$ for which there exists a continuous odd map $A \to S^{k-1}$ for $k\geq p$, but not for $k=p-1$. This defines a \emph{cohomotopic} family $\mathcal{C}_p$, and its sets are characterized by having a orbit space with Lusternik-Schnirelmann category $\geq p+1$ (see Appendix \ref{ap:index}). This family was employed by A. Bahri and P. Lions in \cite{BahriLions} to improve the results of \cite{BahriBerestycki}. The corresponding min-max theorem gives the existence of critical points for $E_\e$ with augmented Morse index $\geq p$, and also a lower bound on the size of the set of such solutions in terms of this topological index.

Using the monotonicity of the cohomological index, one verifies that these families satisfy the following inclusions $\mathcal{M}_p\subset \mathcal{F}_p \subset \mathcal{C}_p$. Furthermore, the proofs in Sections \ref{sec:bounds} and \ref{sec:upperbound} give sublinear bounds from above and from below for the min-max energy values associated to $\mathcal{M}_p$ and $\mathcal{C}_p$, respectively. Hence, these values have the same asymptotic behavior with respect to $p$ as $\e \to 0^+$.
\end{remarq}

\section{Upper Bound} \label{sec:upperbound}
In the next sections we study the asymptotic behavior of $\liminf_{\e \to 0}c_\e(p)$, the limit of $p$-th min-max value of $E_\e$ with respect to the families $\mathcal{F}_p$, as $p \to \infty$. 

The asymptotic behavior of the min-max widths for the area functional has been studied previously by Gromov \cite{Gromov}, Guth \cite{Guth} and Marques-Neves \cite{MarquesNevesInfinite}. The bounds obtained, in this section and the next, may be seen as the analogue of Gromov-Guth bounds for the $p$-widths of $M$ in the context of phase transitions. Our proof is an adaptation to the Sobolev space context of the one presented in \cite{MarquesNevesInfinite}.

In this section we prove the following sublinear upper bound for the min-max values $c_\e(p)$.

\begin{theorem}\label{upperbound} For each $\e>0$ and $p\in\N$, there is a continuous odd map  $$\hat{h} :S^p \to H^1(M)\setminus \{0\}$$ such that $$\frac{1}{2\sigma}\sup_{S^p}E_\e \circ \hat h\leq C p^{1/n},$$ with $C=C(M)>0$. Notice that by the monotonicity of ${\rm Ind}_{\Z/2}$ we have $\hat h_a(S^p) \in \mathcal{F}_p$. In particular, for every $p\in \N$, $$\frac{1}{2\sigma}c_\e(p)\leq C p^{1/n}.$$
\end{theorem} 

To prove this result we adapt Guth's \textit{bend-and-cancel} procedure \cite{Guth} to the context of phase transitions. In doing so, we follow ideas from \cite{MarquesNevesInfinite}, where Marques-Neves adapted Guth's argument to construct $p$-sweepouts of hypersurfaces in a general closed manifold with controlled area. 

Motivated by \cite{Guaraco}, we consider the composition of the one-dimensional solution to the Allen-Cahn equation with distance functions to slices of a $p$-sweepout with low waist. These new functions take values close to $\pm 1$ in a large set and have \textit{jumps} along the slices of the $p$-sweepout. Moreover, it is possible to compute its energies in terms of the areas of the slices. 

Two technical points are worth discussing here before going into the details. Both arise from the nature of the $p$-sweepouts in \cite{MarquesNevesInfinite}. The first one concerns the fact that the continuity of these sweepouts is measured with respect to coefficients modulo 2. In this way, two slices cancel out as they coincide. In such a situation, simply composing with a signed distance function wont produce a continuous family in $H^1(M)$. We account for this by considering \textit{modified} distance functions that smooth out the cancelation of the leaves. The second observation is that computations to estimate the energy of the functions produced do not fit well the $p$-sweepouts presented in \cite{MarquesNevesInfinite}. So, we modify their construction to get slightly \textit{simpler} sweepouts for which computations are easier. More precisely, after identifying $M$ with a cubical complex $K$, we construct  $p$-sweepouts over $K$ with \textit{linear} slices. Then, roughly speaking, we construct our sweepouts for the energy functional on $H^1(K)$, which can be identified with $H^1(M)$. 

\subsection*{Construction of a linear 1-sweepout on $K$ and modified distance functions}

First, remember that any compact smooth manifold can be triangulated. Hence, by \cite{BuchstaberPanov}, Chapter 4, there exists a $n$-dimensional cubical subcomplex $K$ of $I^m$ for some $m$, and a bi-Lipschitz homeomorphism $G:K\to M$. For each $k \in \mathbb{N}$, denote by $c(k)$ the center of the cubes $\alpha \in K(k)_{n}$. 

We need a few preliminary lemmas. The first one is simple and we leave its proof to the reader.

\begin{lemma}\label{directions} For almost every direction $v\in S^{m-1}=\{x\in \R^m : |x|=1\}$ we have:
\begin{enumerate}
\item $v$ is not orthogonal to any cube $\alpha \in K(k)_n$.
\item $\langle v, x \rangle \neq \langle v, y \rangle$, for $x\neq y$ in $c(k)$. \end{enumerate}

Define $f(x)=\langle x ,v \rangle$, with $v$ satisfying (1) and (2). Then,
\begin{enumerate} \setcounter{enumi}{2}
 \item The level sets of $f$ are sections of parallel $(n-1)$-planes on each cube $\alpha \in K(k)_{n}$.  
 \item There exists a small $\rho>0$ such that every level set $f^{-1}(s)$ intersects at most one of the open $n$-cubes centered at points in $c(k)$ with side of length $\rho$.
\end{enumerate}

\end{lemma}

The idea behind Guth's bend-and-cancel argument is to deform a sweepout using a map that projects the complement of a small neighborhood of $c(k)$ onto the $(n-1)$-skeleton $K(k)_{n-1}$. We do this in a way that allow us to control the shape of the new sweepout in every cube $\alpha \in K(k)_{n}$.

If $x$ is the center of a cube $\alpha \in K(k)_{n}$, define $\alpha_r(x)\subset \alpha$ as the $n$-cube centered at $x$ with sides of length $r$ and parallel to $\alpha$, i.e. $\alpha_r(x)=r\cdot(\alpha-x)+x$.

\begin{lemma}
For every  $0<r<1$ there exists a Lipschitz map $F_r:K\to K$ such that 
\begin{enumerate}
\item $F_r(K\setminus \cup_{x\in c(k)}\alpha_r(x))\subset K(k)_{n-1}$.
\item If $\pi \subset \alpha$ is a piece of a $(n-1)$-plane contained in $\alpha$, its image under $F$ is contained in the union of another $(n-1)$-plane with the $(n-1)$-skeleton. More precisely, $F(\pi)\subset \tilde\pi \cup K(k)_{n-1}$, where $\tilde \pi$ is a piece of a $(n-1)$-plane contained in $\alpha$.
\end{enumerate}
\end{lemma}

\begin{proof}
Consider the cubes $\alpha_r=[-r,r]^n$, for $0<r\leq1$, and define $P_r : [-1,1]^n \to [-1,1]^n$ as $$P_r(x)=\begin{cases} \frac{x}{r} & \text{ if }  x\in \alpha_r \\ \frac{x}{\| x\|_{\infty}} & \text{ if }  x\in \alpha \setminus \alpha_r \end{cases},$$ where $\|x\|_\infty=\max\{|x_1|,\dots,|x_n|\}$ for $x=(x_1,\dots,x_n)\in \R^n$. Notice that the restriction of $P_r$ to $\alpha_r$ to is an homothety. In particular, it sends pieces of hyperplanes into pieces of hyperplanes.

For each $\alpha \in K(k)_n$ we pick a linear homomorphism $L_\alpha : \alpha_1 \to \alpha$ such that $L_\alpha(0)=c_\alpha$, where $c_\alpha$ is the center of $\alpha$, and define $$P_{r,\alpha}:\alpha \to \alpha,\ \  \ \ P_{r,\alpha}= L_\alpha  \circ P_r \circ L_\alpha^{-1}.$$ Finally, we define $F_r: K \to K$ by $F_r(x)=P_{r,\alpha}(x)$ whenever $x\in\alpha$. The map is well defined and satisfies the desired properties. 
\end{proof}

Let $f(x)=\langle x ,v \rangle$ where $v$ is one of the directions given by Lemma \ref{directions}. The level sets $f^{-1}(s)$ are piecewise linear closed hypersurfaces. These sets do not necessarily vary continuously on $s$ with respect to the Hausdorff distance in $K$ because they might become empty near vertices that are local maxima or minima. To rule out this possibility it is enough to add a fixed compact set containing $K_0$ in the definition of the sweepouts. We choose to add $K(k)_{n-1}$ because it will simplify other arguments below. Consider the family:
	\[\Sigma_s = F_\rho(f^{-1}(s) \cup K(k)_{n-1}), \quad \mbox{for} \quad s \in \R.\]

\begin{claim}\label{claim1}
Let $\rho>0$ given by Lemma \ref{directions}. The family $\{\Sigma_s\}_{s \in \R}$ varies continuously in the Hausdorff distance. 
\end{claim}

We omit the proof of this claim. This finish the construction of an uniparametric linear sweepout on $K$. From this sweepout we will construct our piecewise linear $p$-sweepouts by combining slices parametrized by the roots of polynomials of degree at most $p$, as in \cite{MarquesNevesInfinite}. However, in contrast to \cite{MarquesNevesInfinite}, we consider complex roots of the polynomials in order to \textit{smooth out} the cancellation of the slices. All this information will be enclosed in the following modified distance functions.

Given $z\in \mathbb{C}$, consider the distance function $$d_z:K\to \R_{\geq 0},\ \ \ d_z(x)={\rm dist}_K(x,\Sigma_{{\rm Re}(z)})+{\rm dist}_{\mathbb{C}}(z,f(K)).$$ This function is the building block of our $p$-sweepouts which we construct in the following way.

For each $a=(a_0,\dots,a_p)\in S^p$, consider the polynomial $P_a(z)=\sum_{i=0}^{p}a_i z^i$ and let $C(a)$ be the set of its roots in the complex plane. We then define the functions $$d_a(x)=\begin{cases} \min \{d_z(x) :z\in C(a)\} & \text{if $C(a)\neq\emptyset$}\\ +\infty & \text{if $C(a)=\emptyset$}\end{cases}.$$ Finally, define $$h_a(x)={\rm sgn}_a(x)d_a(x),$$ where ${\rm sgn}_a(x)={\rm sgn}(P_a\circ f (y))$ for any $y\in F_\rho^{-1}(x)$, whenever $d_a(x)>0$. In Step 1 of Claim \ref{sweepout_upper} we show that ${\rm sgn}_a$ is well defined. Notice that $h_a$ is simply an signed version of $d_a$ which is equivariant on $a$.

In the following two claims we establish several properties of $d_z$ and $h_a$, before presenting the proof of the main result.

Applying Proposition 9.2 from \cite{Guaraco} in every cube $\alpha\in K(k)_n$, we obtain:

\begin{claim} \label{distancia}
$d_z$ is a Lipschitz function with Lipschitz constant 1. Also
\begin{enumerate}
\item $|\nabla d_z|=1$ a.e. on $K$.
\item If $z_n\to z$, then $d_{z_n}\to d_z$ and $\nabla d_{z_n} \to \nabla d_z$ a.e. on $K$.
\end{enumerate}
\end{claim}

\begin{claim} \label{sweepout_upper} Let $\psi_\e$ be the 1-dimensional solution to the Allen-Cahn equation. Then $\hat h_a=\psi_\e \circ h_a \circ G^{-1}$ is a $p$-sweepout.
\end{claim}

\begin{proof}[Proof of Claim  \ref{sweepout_upper}] We divide the proof of this claim into the 4 steps below.

\

\noindent{\bf Step 1.} \textit{$h_a(x)$ is well defined for every $x$}. 

\

If $d_a(x)=0$, there is nothing to check since $h_a(x)=0$.  Then, assume $d_a(x)>0$ and that there are $y,w \in F_\rho^{-1}(x)$ such that $P_a(f(y))$ and $P_a(f(w))$ have different signs. Since the set $F_\rho^{-1}(x)$ is path connected, this implies there is a $q \in K$ such that $F_\rho(q)=x$ and $P_a(f(q))=0$. In particular, $f(q)\in C(a)$ and $x\in \Sigma_{f(q)}$. Then $d_a(x)=0$ which is a contradiction.

\

\noindent{\bf Step 2.} \textit{$h_a$ is a Lipschitz function for every $a\in S^p$}.

\

The fact that $d_a$ is a Lipschitz function follows immediately from the definition. We assert that ${\rm sgn}_a$ is constant in every connected component of $\{x: d_a(x)>0\}$. In fact, take $x\in K$ such that $d_{a}(x)>0$ and assume $P_a\circ f(y)>0$ for every $y \in F_{\rho}^{-1}(x)$. If the same does not hold for every point in a neighborhood of $x$, there would be a sequence $x_n\to x$ in $K$ such that \begin{itemize}\item $d(x_n)>0$ \item $y_n \in F_{\rho}^{-1}(x_n)$ such that $P_a\circ f(y_n)<0$ \item $y_n\to y \in F_{\rho}^{-1}(x)$.\end{itemize} Therefore, $P_a\circ f(y)\leq 0$ which is a contradiction.

Since ${\rm sgn}_a$ is constant in every connected component of $\{x: d_a(x)>0\}$, it is easy to see that $h_a= {\rm sgn}_a \cdot d_a$ is Lipschitz. In fact, take $x$ and $y \in K$. If ${\rm sgn}_a(x)={\rm sgn}_a(y)$ then $|h_a(x)-h_a(y)|\leq|d_a(x)-d_a(y)|\leq {\rm dist}_K(x,y)$. On the other hand, if ${\rm sgn}_a(x)\neq{\rm sgn}_a(y)$ they must belong to different connected components of $\{x: d_a(x)>0\}$. In this case, there exists $z\in K$ such that ${\rm dist}_K(x,y)={\rm dist}_K(x,z)+{\rm dist}_K(z,y)$ and such that $d_a(z)=0$. Then $|h_a(x)-h_a(y)|\leq|d_a(x)+d_a(y)|\leq|d_a(x)-d_a(z)+d_a(y)-d_a(z)|\leq {\rm dist}_K(x,z)+{\rm dist}_K(z,y) = {\rm dist}_K(x,y)$.

\

\noindent{\bf Step 3.} \textit{Let $a_n \to a \in S^p$. Then $\psi_\e \circ h_{a_n} \to \psi_\e \circ h_a$ and $\nabla\psi_\e \circ h_{a_n} \to \nabla\psi_\e \circ h_a$ a.e. on $K$}. 

\

We can assume $C(a)\neq \emptyset$. In fact, this happens only if $a=(\pm1,0,\dots,0)$ and Lemma \ref{raizeshausdorff} implies that $d_{a_n} \to +\infty$ uniformly in that case. Therefore, $\psi_\e \circ h_{a_n}\to \psi_\e \circ h_a=\pm1$ and $\nabla \psi_\e \circ h_{a_n} = \psi_\e'(h_{a_n}) \nabla h_{a_n} \to 0$ a.e. on $K$, by the properties of $\psi_\e$.

By Claim \ref{distancia}, we know that the function $d_z$ satisfies $|\nabla d_z|=1$ a.e. on $K$ and if $z_n\to z$, then $d_{z_n}\to d_z$ and $\nabla d_{z_n} \to \nabla d_z$ a.e. on $K$. The idea is to use Lemmas \ref{truncar}, \ref{raizeshausdorff} and \ref{lipschitz} from the appendix to conclude that a similar statement holds for $d_a$. 

Let $D\in \mathbb{C}$ be an open disc such that \begin{enumerate} \item $f(K) \cup C(a) \subset D$  \item \label{dist} ${\rm dist}_{\mathbb{C}}(w,f(K)) > \mathrm{diam}(M)+ {\rm dist}_{\mathbb{C}}(C(a),f(K))$ for every $w\in \mathbb{C}\setminus D$. \end{enumerate}

Take any sequence $a_n\to a$ in $S^p$. By Lemma \ref{raizeshausdorff} we know $C(a_n)\cap D\to C(a)$ in the Hausdorff distance. This and (\ref{dist}) implies that, for $n$ big enough, ${\rm dist}_{\mathbb{C}}(w,f(K)) > \mathrm{diam}(M)+{\rm dist}_{\mathbb{C}}(C(a_n)\cap D,f(K))$ for every $w\in \mathbb{C}\setminus D$. In particular, there is $z \in C(a_n)\cap D$ such that $d_z < d_w$ for every $w\in \mathbb{C}\setminus D$. More precisely, let $z \in C(a_n)\cap D$ be such that ${\rm dist}_{\mathbb{C}}(z,f(K))$ is minimum. Then, for every $x\in K$,
\begin{align*}
d_z(x) 	&=\mathrm{dist}(x,\Sigma_{\mathrm{Re}(z)})+{\rm dist}_{\mathbb{C}}(z,f(K)) \\
	&\leq \mathrm{diam}(M) +{\rm dist}_{\mathbb{C}}(C(a_n) \cap D,f(K)) \\
	& < {\rm dist}_{\mathbb{C}}(w,f(K)) \leq d_w(x).
\end{align*}
Therefore, to compute $d_{a_n}$ it is enough to take the minimum among the roots $C(a_n)\cap D$, i.e. $d_{a_n}(x)=\min\{  d_z(x) : z\in C(a_n)\cap D\}$. By Lemma \ref{raizeshausdorff} we can label these roots as $\xi_n = (z_1(n),\dots,z_j(n))$ and the roots of $C(a)$ as $\xi=(z_1,\dots,z_j)$, in such a way that $\xi_n \to \xi$. By Claim \ref{distancia}, Lemma \ref{truncar} and Item (3) from Lemma \ref{lipschitz} we conclude that if $a_n\to a$, then $d_{a_n}\to d_a$ and $\nabla d_{a_n} \to \nabla d_a$ a.e. on $K$.

Now, choose $x$ such that $d_a(x)>0$ and choose $y\in F_{\rho}^{-1}(x)$. For $n$ big $d_{a_n}(x)>0$ and since $P_{a_n}(f(y)) \to P_{a}(f(y))$ we must have ${\rm sgn}_{a_n}(x)\to {\rm sgn}_a(x)$. Therefore, $\nabla h_{a_n} \to \nabla h_{a}$ a.e. on $K$ and the statement of Step 3 holds because $\psi_\e$ is a smooth function.

\

\noindent{\bf Step 4.} The function $\hat h_a :S^p \to H^1(M)$ is odd, follows from ${\rm sgn}_a(x)=-{\rm sgn}_{-a}(x)$ whenever $d_a(x)>0$. Finally, that $\hat h_a$ is continuous follows from (2) of Lemma \ref{lipschitz}.

\end{proof}

Combining the results above we can prove the main theorem of this section.

\begin{proof}[Proof of Theorem \ref{upperbound}]

We can estimate the energy of $\hat h_a$ using (1) from Lemma \ref{lipschitz}. 

First, observe that since $|\nabla h_a|\equiv 1$ a.e. we can use the coarea formula to estimate the energies of $\psi_\e \circ h_a$ in each cube $\alpha \in K(k)_{n}$.

\begin{align*}E_\e|_{\alpha}(\psi_\e \circ h_a) & = \int_{\alpha} \frac{\e}{2} |\nabla\psi_\e \circ h_a|^2 + \frac{1}{\e}W(\psi_\e \circ h_a) \ d\mathcal{H}^n \\
& =\int_{-\infty}^{\infty}  \bigg[ \e\frac{\psi_\e'(s)^2}{2} + \frac{W( \psi_\e(s))}{\e}\bigg] \cdot  \mathcal{H}^{n-1}(\{h_a = s\}\cap \alpha)\ ds
\end{align*}

By (1) from Lemma \ref{lipschitz}, there is a constant $C>0$ such that $$E_\e(\hat h_a) = E_\e(\psi_\e \circ h_a \circ G^{-1}) \leq C \sum_{\alpha \in K(k)_n} \int_{\alpha} \frac{\e}{2} |\nabla\psi_\e \circ h_a|^2 + \frac{1}{\e}W(\psi_\e \circ h_a) \ d\mathcal{H}^n$$

And by the computation above we have $$E_\e(\hat h_a) \leq 2C \int_{-\infty}^{\infty}  \bigg[ \e\frac{\psi_\e'(s)^2}{2} + \frac{W( \psi_\e(s))}{\e}\bigg] \cdot  \mathcal{H}^{n-1}(\{x: h_a(x) = s\})\ ds,$$ since $  \sum_{\alpha \in K(k)_n}  \mathcal{H}^{n-1}(\{h_a = s\}\cap \alpha)\  \leq 2  \mathcal{H}^{n-1}(\{h_a = s\}),$ because we might be counting areas in $K(k)_{n-1}$ twice.

To estimate the area $\mathcal{H}^{n-1}(\{x: h_a(x) = s\})$, notice that by definition $$\{x: h_a(x) = s\}  \subset  \{x: \min_{z\in C(a)} d_z(x) =|s|\}.$$ 
The geometry of $\Sigma_{{\rm Re}(z)}$ is simple: by Lemma \ref{directions} it consist of $(H_z\cap \alpha_z)\cup K(k)_{n-1}$, where $H_z\cap \alpha_z$ is the transversal intersection (perhaps empty) of a hyperplane with some $n$-cell $\alpha_z \in K(k)_{n}$. It follows that $\{ x: h_a(x)=|s|  \}$ is contained in the union of the sets
$$ \bigcup_{z\in C(a)}\{x \in \alpha_z:d(x, H_z \cap \alpha_z)=|s| \}$$ and $$\{x \in K:  d(x,K(k)_{n-1})=|s|-\min_{z\in C(a)} d(z,f(K))\}.$$ 

Now, $$\mathcal{H}^{n-1}(\cup_{z\in C(a)}\{x:d(x, H_z \cap \alpha_z)=s\} \leq 2p  C_1 3^{-k(n-1)}$$ and $$\mathcal{H}^{n-1} (\{x:  d(x,K(k)_{n-1})=|s|\}) \leq \mathcal{H}^{n-1} (K(k)_{n-1}) \leq C_2 3^k \mathcal{H}^{n-1}(\partial I^n)$$ where $C_1$ is the maximum area of the intersection of a $(n-1)$-plane with the cube $I^n$ and $C_2$ is the number of $n$-cells in $K$. 

Choosing $3^k \leq p^{1/n}\leq 3^{k+1}$, we have for some constant $C=C(M)>0$ $$ \mathcal{H}^{n-1}(\{ x: h_a(x)=s  \}) \leq 2pC_1 3^{-k(n-1)}+ C_2 3^k \mathcal{H}^{n-1}(\partial I^n) \leq C p^{1/n}.$$  

Finally, we have $$E_\e(\hat h_a)  \leq C p^{1/n} \cdot  \int_{-\infty}^{\infty}  \bigg[ \e\frac{\psi_\e'(s)^2}{2} + \frac{W( \psi_\e(s))}{\e}\bigg] ds \leq 2\sigma C p^{1/n}.$$

Therefore, $$c_\e(p) \leq 2\sigma C p^{1/n}.$$

\end{proof}

\section{Lower bound} \label{sec:bounds}

The main theorem of this section is the following sublinear lower bound for the limit min-max values.

\begin{theorem}\label{thm:lower}
Let $M^n$ be a compact Riemannian manifold. There exists a constant $C=C(M)$ such that the min-max values $c_\e(p)$ satisfy: $$C p^{\frac{1}{n}} \leq \liminf_{\e \to 0^+} c_\e(p)$$ for all $p \in \N$.
\end{theorem}

The following lemma will be a key ingredient in the proof. Roughly speaking, it asserts that the energy $E_\e$ of a function with zero average in a geodesic ball $B_r(x) \subset M$, is comparable to $r^{n-1}$. In what follows $E_\e|_A (u) = \int_A \e|\nabla u|^2/2 + W(u)/\e,$ for $A\subset M$ measurable.

\begin{lemma}\label{zeromean}
There exist $r_0=r_0(M)>0$ and $c_1=c_1(M,W)>0$ such that $$E_\e|_{B_r}(u) \geq c_1 r^{n-1},$$ whenever $\e \leq r \leq r_0$, $B_r=B_r(x)$ for some $x\in M$, $|u|\leq 1$ and $\int_{B_r} u = 0$.
\end{lemma}

Let us show first how Lemma \ref{zeromean} implies Theorem \ref{thm:lower} and postpone its proof to the end of the section.

\begin{proof}[Proof of Theorem \ref{thm:lower}]

Recall the following fact: there is a positive constant $\nu = \nu(M)>0$ such that, for all $p \in \N$, we can find $p$ disjoint closed geodesic balls $B_1,\ldots, B_p \subset M$ of radius $r_p = \nu p^{-\frac{1}{n}}$ (compare with \cite[\S 3]{Guth}). 
Here $B_i=B_i(x_i)$ for some $x_i\in M$ and we assume $\nu$ is smaller than the $r_0$ given by Lemma \ref{zeromean}, in particular $r_p\leq r_0$. 

The following version of the Borsuk-Ulam property for the cohomological index, implies that for every $A \in \mathcal{F}_p$, there exists $u\in A$ such that $\int_{B_i} u =0$, for every $i=1,\dots,p$.

\begin{lemma} \label{lem:lb1}
Given a paracompact $\Z/2$-space $A$ with $\ind_{\Z/2}(A)\geq p+1$, every continuous equivariant function $f:A \to \R^p$ has a zero, i.e., $f^{-1}(0) \neq \emptyset$.
\end{lemma}

\begin{proof}
Suppose, by contradiction, that we can find $f:A \to \R^p$ continuous, equivariant and such that $f^{-1}(0)=\emptyset$. Define $\varphi: A \to S^{p-1}$ by $\varphi(x) = \frac{x}{||x||}$, for ${x \in A}$. Then $\varphi$ is a continuous equivariant map, and thus
	\[p+1\leq \ind_{\Z/2}(A) \leq \ind_{\Z/2}(S^{p-1}) = p,\]
which is a contradiction.
\end{proof}

\begin{remark} The previous lemma holds more generally for any topological $\Z/2$-index, for it holds $\ind_{\Z/2}(S^p) \leq p+1$ for every such index. See the Appendix \ref{ap:index} for more details.
\end{remark}

By replacing $A$ with $ \tau(A)$ if necessary (where $\tau:H^1(M) \to H^1(M)$ is the truncation map defined in the proof of \ref{thm:minmax} ) we might assume that $|u|\leq 1$, for every $u\in A$, since $\tau (A) \in \mathcal{F}_p$ and $\sup E_\e(A) \geq \sup E_\e(\tau(A))$.

Now choose $\e$ so that $0<\e \leq r_p \leq r_0$. Lemma \ref{lem:lb1} implies the existence of $u\in A$ with $\int_{B_i} u =0$. Finally, Lemma \ref{zeromean} implies 
	\[E_\e(u) \geq \sum_{i=1}^{p} E_\e|_{B_i}(u) \geq c_1pr_p^{n-1} = c_1\nu^{n-1}p^{1/n}.\]
Hence, we have $\max_{u \in A} E_\e(u) \geq C p^{1/n}$, for every $A \in \mathcal{F}_p$ and $\e \in (0,r_p)$, where $C=c_1\nu^{n-1}$. In particular, for every $\e \in (0,r_p)$ we have $$C p^{1/n} \leq c_\e(p).$$ This proves Theorem \ref{thm:sublinear}.

\end{proof}

\begin{proof}[Proof of Lemma \ref{zeromean}]

By the compactness of $M$ and a comparison argument, we may assume that $r_0=r_0(M)$ is such that we can find a constant $c=c(M)>1$ such that
	\begin{equation}\label{ball} \frac{1}{c} r^n \leq \hm^n(B_r(x)) \leq cr^n, \quad \mbox{for all} \quad x \in M, \ r \in (0,r_0).\end{equation}
Denote $|A|=\hm^n(A)$ for $A \subset M$ and, for a fixed $a \in (0,1)$, $$\begin{array}{c} A_{+} = \{ x \in B_r : a \leq u \}  \\ A_0 = \{ x \in B_r : -a < u < a \} \\ A_{-} = \{ x \in B_r : u \leq - a \}   \end{array}.$$

The lemma is a consequence of the interplay between two inequalities. The first one is an isoperimetric inequality due to De Giorgi (see \cite[Lemma 1.4]{CaffarelliVasseur}). Roughly speaking, it states that functions in $H^1(M)$ cannot have jump singularities. 
\begin{lemma} \label{lem:dg-ipi}
There exist $c_0=c_0(M)>0$ and $r_0=r_0(M)>0$ such that for every $u \in H^1(M)$, for all real numbers $a_0<b_0$, for every $r \in (0,r_0)$, and for all $x \in M$, we have \begin{equation}\label{degiorgi} \left|\{u\leq a_0\} \cap B_r(x)\right|\cdot \left|\{u\geq b_0\} \cap B_r(x)\right|^{1-\frac{1}{n}} \leq \frac{c_0 r^n}{b_0-a_0} \int_{\{a_0<u<b_0\} \cap B_r(x)} |\nabla u|.\end{equation}
\end{lemma}

The second inequality (which proof we omit since it is obtained as in \cite[\S 6]{Guaraco}) is \begin{equation}\label{ineq2}|A_{\pm}| \geq \frac{a}{2}|B_r| - W(a)^{-1} \e E_\e|_{B_r}(u) \geq  \frac{a}{2c}r^{n} - W^{-1}(a) \e E_\e|_{B_r}(u).\end{equation}
It is a consequence of $ W(a)|A_0| \leq \e E_\e|_{B_r}(u)$ (which follows directly from the definitions of $W$ and $E_\e$), the fact that $|u|\leq 1$ with $\int_{B_r} u =0$ and \eqref{ball}. Notice that we can assume $ \frac{a}{2c}r^{n} - W^{-1}(a) \e E_\e|_{B_r}(u) \geq 0.$ Otherwise $\e<r$ and (\ref{ball}) would imply $E_\e|_{B_r}(u) > aW(a)(2c)^{-1}\cdot r^{n-1}$, which is what we want to prove.

Combining (\ref{degiorgi}) and Lemma \ref{lem:dg-ipi} (with $-a_0=b_0=a$) we obtain $$\bigg|\frac{a}{2c}r^n -W(a)^{-1} \e E_\e|_{B_r}(u)\bigg|^{2-1/n} \leq \frac{c_0}{2a} r^n \int_{A_0}|\nabla u|.$$ Moreover, \begin{align} \label{eq:lb2}
		\int_{A_0} |\nabla u| & \leq |A_0|^{1/2} \left( \int_{B_r}|\nabla u|^2 \right)^{1/2} \nonumber\\
		& \leq  \bigg(W(a)^{-1}\e E_\e|_{B_r}(u)\bigg)^{1/2}   \left( \frac{2}{\e} E_\e|_{B_r}(u)\right)^{1/2}\\
		& \leq \left( \frac{2}{W(a)} \right)^{1/2} E_\e|_{B_r(x)}(u). \nonumber
	\end{align}
Now (\ref{eq:lb2}) and (\ref{ball}) imply $$\bigg|\frac{a}{2c}r^n -W(a)^{-1} \e E_\e|_{B_r}(u)\bigg|^{2-1/n} \leq c_2 r^n E_\e|_{B_r}(u),$$ where $c_2=c_0(a\sqrt{2W(a)})^{-1}$ and since $\e/r\leq 1$ by hypothesis, we conclude $$\bigg|\frac{a}{2c} -W(a)^{-1} \frac{E_\e|_{B_r}(u)}{r^{n-1}}\bigg|^{2-1/n} \leq c_2 \frac{E_\e|_{B_r}(u)}{r^{n-1}}.$$

This inequality is of the form $|A- B s|^{2-1/n}\leq C s$, where $s=E_\e|_{B_r}(u)/r^{n-1}$ and $A,B$ and $C$ are positive constants depending only on $M$ and $W$. Since it does not hold true when $s=0$, it implies that $s$ cannot be arbitrarily small. In particular, there exists $c_1=c_1(A,B,C)>0$ such that $s\geq c_1$, and the lemma follows.

\end{proof}

\section{Comparison with Marques-Neves \texorpdfstring{$p$}{p}-widths} \label{sec:comp}
In this section, we show that the min-max values for the energy functional and cohomological families are bounded below, as $\e \to 0^+$, by the corresponding $p$-widths $\omega_p(M)$ of Almgren-Pitts min-max theory, defined in terms of high-parameter families of sweepouts in \cite{MarquesNevesInfinite}. More precisely we prove

\begin{theorem} \label{thm:wplp}
For every $p \in \N$, it holds
	\[\omega_p(M) \leq \frac{1}{2\sigma}\liminf_{\e \to 0^+} c_\e(p).\]
\end{theorem}

As we will see later, the $p$-widths, $\omega_p(M)$, are defined in \cite{MarquesNevesInfinite} in terms of maps $\Phi: X \to\znm$ where $X$ is a cubical complex and $\znm$ is the space of mod 2 integral $(n-1)$-cycles in $M$ with zero boundary. However, the min-max values $c_\e(p)$ are defined in terms of the elements of $\mathcal{F}_p$ which can be very different from continuous images of cubical complexes. Because of this, in order to prove Theorem \ref{thm:wplp} we will first approximate a set $A \in \mathcal{F}_p$ which is almost optimal (in the sense that its energy is close to $c_\e(p)$) by the image of an odd map $h$ from a $p$-dimensional cubical complex into $H^1(M)$. 

In what follows, we will use the notation for cubical complexes discussed in the \nameref{notation} section of the Introduction.

\subsection{Cubical subcomplexes and min-max values} Initially, we need to show that the min-max value $c_\e(p)$ can be obtained by restricting ourselves to sets which are the image of certain $p$-dimensional subcomplexes of $Q(m,k)$ by odd maps into $H^1(M)$.

Fix $p \in \N$ and denote by $\mathcal{C}_p$ the family of all $X$ that are $p$-dimensional symmetric cubical subcomplexes of $Q(m,k)$, for some $m,k \in \N$, with $\ind_{\Z/2}(X)\geq p+1$. For every such $X$, we consider also the family $\Gamma(X)$ of all continuous odd maps $h:X \to H^1(M)/\{0\}$ and its associated min-max values
	\[c_\e(X) = \inf_{h \in \Gamma(X)} \sup_{h(X)} E_\e.\]
By the monotonicity property of the index, we have $h(X) \in \mathcal{F}_p$ for all $h \in \Gamma(X)$, thus $c_\e(p)\leq c_\e(X)$. Moreover, we have (compare with \cite[Lemma 4.7]{MarquesNevesInfinite} and \cite[\S 1.5]{MarquesNevesIndex})

\begin{lemma} \label{lema_pdim}
For all $p \in \N$, it holds
	\[c_\e(p) = \inf_{X \in \mathcal{C}_p} c_\e(X).\]
\end{lemma}

\begin{proof}
Given $\delta>0$, let $A_0 \in \mathcal{F}_p$ be such that $\sup E_\e(A_0) \leq c_\e(p) + \delta/2$. Given an arbitrary neighborhood $U$ of $A_0$ in $H^1(M) \setminus \{0\}$, we can find a subspace $E\subset H^1(M)$ with $m:=\dim E<+\infty$ and $A \subset U \cap E$ such that $\ind_{\Z/2}(A)=\ind_{\Z/2}(A_0)$ (see \cite[Proposition 3.1]{Marzocchi}). We identify $E$ with $\R^m$ by a linear isomorphism $T:\R^m \to E$ such that $T(Q^m)$ is a cube in $H^1(M)$ containing $A$ in its interior. Under this identification, choose $k \in \N$ such that $\alpha \subset U$ for every $m$-cell $\alpha \in Q(m,k)_m$ with $\alpha \cap A \neq \emptyset$. If $X_m$ is the union of all such cells, then $A \subset X_m \subset U$ and thus $\ind_{\Z/2}(X_m)\geq p+1$, provided we choose $U$ so that $\ind_{\Z/2}(\overline U)=\ind_{\Z/2}(A)$. Let $X$ be the $p$ skeleton of $X_m$, that is, the union of all $p$-cells of $X_m$. We claim that $X \in \mathcal{C}_p$. If $U$ also satisfies $\sup E_\e(U) \leq \sup E_\e(A_0) + \delta/2$, then it will follow that
	\[c_\e(X) \leq \sup E_\e(U) \leq \sup E_\e(A_0) + \frac{\delta}{2} \leq c_\e(p) + \delta.\]
Hence, $\inf_{X \in \mathcal{C}_p} c_\e(X) \leq c_\e(p)$. In order to show that $X$ has index $\geq p+1$, it suffices to observe that $H^p(X_m,X;\Z/2) = 0$. The exactness of the cohomology sequence of the pair $(X_m,X)$ implies then that the inclusion ${X \hookrightarrow X_m}$ induces an injective morphism $H^p(X_m,\Z/2) \to H^p(X,\Z/2)$. Therefore, we have $\ind_{\Z/2}(X) \geq \ind_{\Z/2}(X_m) \geq p+1$ and $X \in \mathcal{C}_p$.
\end{proof}

\subsection{$p$-widths} Now we are ready to present the definitions of $p$-sweepouts and the $p$-widths, $\omega_p(M)$, following \cite{MarquesNevesInfinite}.

Let $X$ be a cubical subcomplex of $Q(m,k)$ and $\Phi:X \to \znm$ be a continuous map in the flat metric. We say that $\Phi$ is a \emph{$p$-sweepout} if
	\[\Phi^*(\beta^p)\neq 0 \ \mbox{in} \ H^p(X,\Z/2)\]
for some non-trivial cohomology class $\beta \in H^1(\znm, \Z/2)$. 

We recall that the first cohomology group of each connected component of $\znm$, with $\Z/2$-coefficients, is isomorphic to $\Z/2$. In fact, in the next subsection we present the more complete description of the cohomology groups of $\znm$ with $\Z/2$-coefficients, as discovered by Almgren \cite{Almgren}.

A cubical subcomplex $X$ of $Q(m,k)$ is said to be \emph{$p$-admissible} if there exists a $p$-sweepout $\Phi:X\to \znm$ that has \emph{no concentration of mass}, i.e.
	\[\lim_{r \to 0^+} \sup \left\{ ||\Phi(x)||(B_r(p)) : x \in \mathrm{dmn}(\Phi), p \in M \right\} = 0.\]

We remark that a map into $\znm$ which is continuous in the mass norm has no concentration of mass (see \cite[Lemma 3.8]{MarquesNevesInfinite}). The set of all $p$-sweepouts with no concentration of mass will be denoted by $\mathcal{P}_p$. 

Define the \emph{$p$-width of $M$} as
	\[\omega_p(M) = \inf_{\Phi \in \mathcal{P}_p} \sup_{x \in \mathrm{dmn}(\Phi)} {\bf M}(\Phi(x)).\]
Arguing as in \cite[\S 1.5]{MarquesNevesIndex} (or as in the proof of Lemma \ref{lema_pdim}) we see that it suffices to consider $p$-sweepouts defined in $p$-dimensional cubical complexes, i.e.
	\[\omega_p(M) = \inf\left\{\sup_{x \in \mathrm{dmn}(\Phi)}{\bf M}(\Phi(x)) : \Phi \in \mathcal{P}_p, \dim \mathrm{dmn}(\Phi)=p\right\}.\]

\subsection{Proof of Theorem \ref{thm:wplp}}

Theorem \ref{thm:wplp} is a direct consequence of the following approximation result:

\begin{theorem} \label{thm:comp}
Fix $\tilde \sigma \in (0,\sigma/2)$. There exist positive constants $C=C(p,M)$ and $\delta_0=\delta_0(p,M)$ with the following property. Given $\delta_1 \in (0,\delta_0)$, we can find $\e_0=\e_0(\delta_1) \in (0,\delta_1)$ such that for all $\e \in (0,\e_0)$ and every $X \in \mathcal{C}_p$ with $c_\e(X) \leq c_\e(p)+\e$, there exists an even map $\Phi:X \to \znm$ which is continuous with respect to the mass norm and satisfies
	\[\sup_{x \in X} {\bf M}(\Phi(x)) \leq \frac{c_\e(p)+\e}{4\tilde \sigma} + C\delta_1.\]
Moreover, the map $\tilde \Phi:\tilde X \to \znm$ induced by $\Phi$ in the orbit space $\tilde X=X/\{x\sim -x\}$, is a $p$-sweepout. 
\end{theorem}

Let us show now how this result implies Theorem \ref{thm:wplp}.

\begin{proof}[Proof of Theorem \ref{thm:wplp}]
Observe that
	\[\omega_p(M) \leq \frac{c_\e(p)+\e}{4\tilde \sigma}+(p + c_0(p,M)b_p) \delta_1.\]
for all $\delta_1 \in (0,\min\{\nu_M,c(p,M)\})$, $\e \in (0,\e_0(\delta_1))$ and $\tilde \sigma \in (0,\sigma/2)$. Hence, making $\delta_1 \downarrow 0$ (which implies $\e_0(\delta_1) \downarrow 0$), we obtain
	\[\omega_p(M) \leq \frac{\liminf_{\e \to 0^+} c_\e(p)}{4\tilde \sigma},\]
and finally, as $\tilde \sigma \uparrow \sigma/2$,
	\[\omega_p(M) \leq \frac{1}{2\sigma} \liminf_{\e \to 0^+} c_\e(p).\]

\end{proof}

The following subsections comprise the proof of Theorem \ref{thm:comp}, which follows ideas from \cite{Guaraco}. Here we present a sketch of its proof.

Let $X \in \mathcal{C}_p$ with $c_\e(X) \leq c_\e(p)+\e$. Ideally, given $h\in \Gamma(X)$, we would like to select a level set of $h(x) \in H^1(M)$ for each $x\in X$, in such a way that they form a map $X\to \znm$ continuous in the mass norm, with no concentration of mass and with mass controlled by $c_\e(p)+\e$. The problem with this is that continuity in $H^1(M)$ does not even imply continuity in the mass norm for the level sets. However, we can still show that, roughly speaking, the level sets of $h(x)$ vary continuously with respect to the flat norm. In \cite{MarquesNevesInfinite}, in order to pass from maps which are continuous in the flat norm to maps that are continuous in the mass norm, the additional condition of \textit{no concentration of mass} is required. In the same article they conjecture that such additional condition might not be necessary. Fortunately to us, this has been recently proved in X. Zhou \cite{Zhou2015}.

We apply a result from \cite{Zhou2015} in a discrete setting. More precisely, Theorem \ref{itp} allow us to interpolate a discrete map defined only on the vertices of $X(k)_0$ (for some $k$) which is fine in the flat norm, to a map on $X(l)_0$ (for some bigger $l$) which is fine in the mass norm. Then we can apply another interpolation result from \cite{MarquesNevesInfinite} that allow us to construct a continuous extension of this map to the whole cubical complex $X$, with controlled mass. 

Finally, to verify that the map obtained after the interpolations is topologically non-trivial, we relate the cohomological index ${\rm Ind}_{\Z/2}$ with the cohomology of $\znm$. We do this by realizing the set of integral flat chains modulo $2$ as an orbit space of a free $\Z_2$-space. To this end, we rely on the fact that the homotopy groups of $\znm$ are the same of the infinite dimensional projective space $\R\p^\infty$, as proven by Almgren in \cite{Almgren}.

\subsection{Interpolation: discrete to continuous}

Recall that the \emph{fineness} of a map $\phi:X(j)_0 \to \znm$ is defined by
	\[{\bf f}(\phi) = \sup\left\{ {\bf M}(\phi(x)-\phi(y)) : x,y \in X(j)_0 \ \mbox{adjacent vertices} \right\}.\]
This notion can be thought as the discrete counterpart of the modulus of continuity of a map into $\znm$ with respect to the mass norm. Similarly, we can consider the fineness of a discrete map with respect to the flat metric by replacing the mass with $\mathcal{F}$ in the definition above.

The next theorem, which follows from \cite[Theorem 14.1]{MarquesNevesWillmore} (see also \cite{MarquesNevesInfinite}) allows us to obtain maps into $\znm$ which are continuous in the mass norm from a discrete map with small fineness.

\begin{theorem} \label{Alm_ext}
Given $p \in \N$, there exist constants $C_0=C_0(p,M)>0$ and $\delta_0=\delta_0(M)>0$ with the following property: given a $p$-dimensional cubical subcomplex $X$ of some $Q(m,k)$ and a discrete map $\phi:X_0 \to \znm$ with ${\bf f}(\phi)<\delta_0$, there exists a map $\Phi:X \to \znm$ which is continuous with respect to the mass norm satisfying:
	\begin{enumerate}
		\item $\Phi$ extends $\phi$, that is, $\Phi|_{X_0} = \phi$.
		\item For every $j$ and every $\alpha \in X_j$, the restriction of $\Phi$ to $\alpha$ depends only on the values of $\phi(x)$ for $x \in \alpha_0$.
		\item For all $x,y \in X$ which lie in a common $p$-cell of $X$, it holds
			$${\bf M}(\Phi(x)-\Phi(y)) \leq C_0 \ {\bf f}(\phi).$$
	\end{enumerate}
\end{theorem}

Following \cite{MarquesNevesInfinite}, the map $\Phi$ given by the theorem above will be called the \emph{Almgren extension} of $\phi$. We emphasize that the constant $C$ above depends only on the dimension of the cubical complex $X$, and not on $m$.

\subsection{Almgren's Isomorphism}
In the seminal paper \cite{Almgren}, F. Almgren proved that the $i$-th homotopy group of the space of mod 2 $k$-dimensional integral flat chains in $M$, $\mathcal{Z}_k(M,\Z/2)$, with the flat topology is isomorphic to the $(k+i)$-th homology group of $M$, with $\Z/2$ coefficients, i.e.
	\[\pi_i(\mathcal{Z}_k(M,\Z/2),\{0\}) \backsimeq H_{k+i}(M,\Z/2),\]
for all $i$. We denote by
	\[F_A:\pi_1(\znm,\{0\}) \to H_n(M,\Z/2)\]
the corresponding isomorphism for $k=n-1$ and $i=1$. Since $H_n(M,\Z/2)$ is isomorphic to $\Z/2$, and $H_{n+i}(M,\Z/2)$ are trivial for $i\geq 1$, Almgren's result shows that the path connected component of $\znm$ containing $0$ has the same homotopy groups of the infinite dimensional real projective space $\R\p^\infty$.

\begin{remarq} \label{rmk:componentes}
For $i=0$, we get a bijection between $\pi_0(\znm)$, which can be identified with the set of path connected components of $\znm$, and $H_{n-1}(M,\Z/2)$. In particular, if this homology group is non-trivial, then $\znm$ is not path connected. Nevertheless, all path connected components of the space of chains are isometric, since it is a topological group with respect to the flat metric and the translations are isometries. Furthermore, from the description of Almgren's isomorphism given below, we see that it is possible to extend $F_A$ to the fundamental group of $\znm$ with base point in any path connected component. Hence, we can omit the reference to the base point in the fundamental group of $\znm$, keeping in mind that it refers to the fundamental group of a certain path connected component.
\end{remarq}
	
The isomorphism $F_A$ can be explicitly described as follows. There are constants $\nu_M>0$ and $\rho_M>0$ such that for every cycle $T \in \znm$ with $\mathcal{F}(T)<\nu_M$, there exists an integral $\Z/2$ current $S \in {\bf I}_n(M)$ such that $\partial S = T$ and $\M(S)\leq \rho_M \mathcal{F}(T)$. Such current $S$ is called an \emph{isoperimetric choice} for $T$. It is possible to show that, if we choose $\nu_M$ small enough, this choice is actually unique (see \cite{MarquesNevesInfinite}). 

Given a map $\phi:S^1 \to \znm$ continuous in the flat topology, we choose $k \in \N$ so that
	\[\mathcal{F}(\phi(x_{j+1})-\phi(x_{j})) \leq \nu_M \ \mbox{for} \ j=0,\ldots, 3^k-1,\]
where $x_j = e^{2\pi i \cdot j3^{-k}}$. If $A_j$ is the isoperimetric choice for $\phi(x_{j+1})-\phi(x_j)$, then Almgren defines
	\[F_A([\phi]) = \left[ \sum_{j=0}^{3^k-1}A_j \right] \in H_n(M,\Z/2).\]
We say that $\phi$ is a \emph{sweepout} if this homology class is non-trivial, which amounts to $\phi$ being a homotopically non-trivial path in $\znm$.

If $Z$ is a path connected component of $\znm$, then by Hurewicz Theorem and Remark \ref{rmk:componentes} we see that the first homology group (with integer coefficients) of $Z$ is isomorphic to $\Z/2$. By the Universal Coefficient Theorem for cohomology, it follows that
	\[H^1(Z,\Z/2) \backsimeq \Z/2, \quad \mbox{for all} \quad Z \in \pi_0(\znm).\]


\begin{remarq} \label{rmk:eqswo} As pointed in \cite{MarquesNevesInfinite}, a continuous map $\Phi:X \to \znm$ defined on a cubical complex $X$ is a $p$-sweepout if and only if we can find a cohomolgy class $\beta \in H^1(X,\Z/2)$ with $\beta^p\neq 0$ and the following property: given a cycle $\gamma:S^1 \to X$, we have $\beta[\gamma] \neq 0$ iff $\Phi \circ \gamma$ is a sweepout. By Remark \ref{rmk:componentes} and the description of $F_A$ given above, we see that this fact holds for the other connected components of $\znm$ as well.
\end{remarq}

\begin{remarq} The lower bound in Theorem \ref{thm:wplp} may be described more precisely in terms of sweepouts which detect the non-trivial cohomology class ${\lambda \in H^1(\znm,\Z/2)}$ of the path component of $0$ in $\znm$. In other terms, we will prove that $c_\e(p)$ is bounded below by $2\sigma \cdot \omega(\lambda^p,M)$, as $\e \to 0^+$, where $\omega(\lambda^p,M) \geq \omega_p(M)$ is the min-max value associated to the family of sweepouts which detect $\lambda^p$.
\end{remarq}

\subsection{${\rm Ind_{\Z/2}}$ and non-trivially of $p$-sweepouts}\label{top-swo} In this subsection we establish a relation between the cohomological index ${\rm Ind_{\Z/2}}$ and the notion of $p$-sweepout, which will serve as a criterion for checking the non-triviality of the $p$-sweepouts.

For this purpose, we compare non-trivial cohomology classes in \linebreak $H^1(\R\p^\infty,\Z/2)$ and $H^1(\znm,\Z/2)$. We do this in terms of equivariant maps by realizing $\znm$ as the orbit space of a free $\Z/2$-space. From the discussion presented after the definition of Almgren's Isomorphism, we observe that the path connected components of $\znm$ are Eilenberg-Mac Lane spaces of type $K(\Z/2,1)$, which means its homotopy groups are null except the first which is isomorphic to $\Z/2$ (see \cite{Spanier}). Hence a natural candidate for this $\Z/2$-space is the universal covering space of one of the connected components of $\znm$. To guarantee that this covering space exists, we need to verify that $\znm$ is locally path connected and semi-locally simply connected. The latter follows directly from \cite[Corollary 3.6]{MarquesNevesInfinite}, while the former is a consequence of the results of \cite{Almgren}, as we state below.

\begin{lemma} The space $\znm$ is locally path connected with respect to the flat topology, that is, the path connected components of every open set $U \subset \znm$ in the flat topology are open.
\end{lemma}

\begin{proof}
It suffices to verify the result for the open balls $U=B_r^{\mathcal{F}}(T)$ in the flat metric, for every $T \in \znm$ and $r>0$. Let $C \subset U$ be a path connected component of $U$ and let $S \in C$. It follows from \cite[Theorem 8.2]{Almgren} that we can find a $r_1>0$ with the following property: for every $S' \in B_{r_1}^{\mathcal{F}}(S)$, there exists a continuous path $\alpha:[0,1] \to B_r^{\mathcal{F}}(T)$ such that $\alpha(0)=S$ \and $\alpha(1)=S'$. This shows that $S$ and all $S' \in B_{r_1}^{\mathcal{F}}(S)$ belong to the same path connected component of $U$, that is, $B_{r_1}^{\mathcal{F}}(S) \subset C$. Since $S \in C$ is arbitrary, this shows that $C$ is open.
\end{proof}

We can now state and prove the main result of this subsection.

\begin{proposition} \label{prop:pswo}
Let $X$ be a symmetric cubical subcomplex of $Q(m,k)$ with $\ind_{\Z/2}(X)\geq p+1$, for some $m,k \in \N$, and $\Phi:X \to \znm$ be a continuous map in the flat topology. Suppose that $\Phi$ is even and consider the induced map $\tilde \Phi:\tilde X \to \znm$ such that $\tilde \Phi = p \circ \Phi$, where ${p:X \to \tilde X = X/\{x \sim -x\}}$ is the orbit map. If the induced homomorphism ${\tilde \Phi_*:\pi_1(\tilde X) \to \pi_1(\znm)}$, satisfies $\ker \tilde\Phi_* = p_*\pi_1(A)$, then $\tilde\Phi$ is a $p$-sweepout.
\end{proposition}

\begin{proof}
We can assume that $X$ is path connected, since one of its path connected components must have cohomological index $\geq p+1$. Denote by $Z$ the path connected component of $\znm$ containing $\Phi(X)$. From the previous lemma and Corollary 14 of \cite[\S 2.5]{Spanier}, we see that there exists a covering map $\pi:E \to Z$ with $\pi_1(E)=0$. Since the morphism
	\[\Phi_*=\tilde\Phi_* \circ p_*: \pi_1(\tilde X) \to \pi_1(Z)\]
is trivial, we can lift $\Phi$ to a continuous map $F:X \to E$, so that $\pi \circ F = \Phi$. We regard $E$ as a $\Z/2$-space with the natural action of $\pi_1(Z) \backsimeq \Z/2$.

We assert that the map $F$ is equivariant. In fact, choose a path $\alpha:[0,1] \to X$ such that $\alpha(1)=-\alpha(0)$. Then $[p\circ \alpha]$ is non-trivial in $\pi_1(\tilde X,p\circ\alpha(0))$ and does not belong to $p_*\pi_1(X,\alpha(0))$, which shows that $\tilde \Phi_*[p\circ \alpha]=[\Phi \circ \alpha]$ is nonzero in $\pi_1(Z)$. We will prove that the action of $\left[p\circ \alpha\right]$ on $X$ induced by the action of $\pi_1(\tilde X)$ on $X$ agrees with the antipodal map. In fact, let $x \in X$ and choose a path $\gamma:[0,1] \to X$ joining $x$ to $\alpha(0)$. We write $\tilde \gamma = p\circ \gamma$ and $\tilde \beta = \tilde\gamma^{-1} * (p\circ \alpha) * \tilde \gamma$. The lift $\beta$ of $\tilde\beta$ with respect to $p$ satisfies $\beta(1) \in p^{-1}(p(\beta(1)))=\{x,-x\}$ and $\beta(1) = \left[p\circ \alpha\right] \cdot x$. Noting that the action of $\pi_1(\tilde X)$ on $X$ is free, we get $\left[ p \circ \alpha \right] = -x$. To conclude that $F$ is equivariant, we observe that $F\circ \beta$ is a lift of $\tilde\Phi(\tilde \beta) = (\pi \circ F \circ \gamma)^{-1} * (\Phi \circ \alpha) *(\pi \circ F \circ \gamma)$ and therefore $$F(-x) = F(\beta(1)) = [\Phi \circ \alpha] \cdot F(\gamma(0)) = [\Phi \circ \alpha] \cdot x.$$


Now let $\xi:E \to S^\infty$ be a continuous equivariant map, so that $\tilde \xi: Z \to \R\p^\infty$ is a classifying map for the $\Z/2$-bundle $E \to Z$. From the claim above, we see that the map $\tilde f = \tilde \xi \circ \tilde \Phi$ is a classifying map for $X \to \tilde X$. If $w$ is the non-trivial cohomology class in $H^1(\R\p^\infty,\Z/2)$, then $\lambda := \tilde f^{*}w \in H^1(\tilde X,\Z/2)$ satisfies $\lambda^p \neq 0$. We will show that given a path $\gamma:S^1 \to \tilde X$, we have $\lambda[\gamma] \neq 0$ if, and only if, $\tilde \Phi \circ \gamma$ is a sweepout. By Remark \ref{rmk:eqswo}, this proves that $\tilde\Phi$ is a $p$-sweepout. 

Given such a path, we have
	\[\tilde f^*(w)[\gamma] = \lambda[\gamma] \neq 0 \iff w(\tilde\xi_*[\tilde\Phi \circ \gamma]) \neq 0 \iff \tilde\xi_*[\tilde\Phi \circ \gamma] \neq 0 \ \mbox{in} \ H_1(\R\p^\infty).\]
Using Hurewicz Theorem, one easily verifies that $\tilde \xi$ induces an isomorphism $H_1(Z) \to H_1(\R\p^\infty)$ and
	\[\lambda[\gamma] \neq 0 \iff [\tilde\Phi \circ \gamma] \neq 0 \ \mbox{in} \ H_1(Z) \iff [\tilde \Phi \circ \gamma] \neq 0 \ \mbox{in} \ \pi_1(Z),\]
This shows that $\lambda[\gamma] \neq 0$ if, and only if, $\Phi \circ \gamma$ is a sweepout, and the claimed result.
\end{proof}

\subsection{Construction of a discrete map fine in the flat norm}  In this subsection we show how to obtain discrete even maps into $\znm$ which are fine in the flat metric, from maps in $\Gamma(X)$. We do this for sufficiently small $\e>0$ and almost optimal complexes $X \in \mathcal{C}_p$. More precisely, we choose finite perimeter sets $\{h(x)>s_x\}$ for each vertex $x \in X(k)_0$, in a sufficiently fine subdivision of $X$, in such a way that for each pair of adjacent vertices the perimeter of these sets are close in the flat metric. This gives us a first discrete approximation of a $p$-sweepout.

Fix $\tilde \sigma \in (0,\sigma/2)$, where $\sigma=\int_{-1}^{1}\sqrt{W(s)/2}\,ds$, and let $h\in\Gamma(X)$ for some $X \in \mathcal{C}_p$ which is a cubical subcomplex of $Q(m,k)$, with $m,k \in \N$. For each $x \in X$ write $h_x=h(x)\in H^1(M)$ and consider its normalization given by $\tilde h_x=\Psi \circ h_x$, where $\Psi:\R \to \R$ is the function
	\[\Psi(t) = \int_0^t \sqrt{W(s)/2}\,ds.\]
Notice that $\Psi$ takes values in the interval $[-\sigma/2,\sigma/2]$. This normalization is intended so that the $BV$-norm of $\tilde h_x$ is bounded by the energy of $h_x$. More precisely we have
\begin{equation}\label{normalization}
	|\nabla \tilde h_x| = \sqrt{W(h_x)/2} \cdot |\nabla h_x| \leq \frac{1}{2}\left(\frac{\e|\nabla h_x|^2}{2} + \frac{W(h_x)}{\e} \right).
\end{equation}
	
For every $x\in X$, there exists $\tilde s_x \in [-\tilde \sigma, \tilde \sigma]$ for which $\{\tilde h_x>\tilde s_x\}$ is a set of finite perimeter satisfying
	\[2\tilde \sigma \M(\partial \llbracket \{\tilde h_x > \tilde s_x\}\rrbracket) = 2\tilde\sigma ||\partial \{\tilde h_x > \tilde s_x\}||(M) \leq \int_{-\tilde \sigma}^{\tilde \sigma} ||\partial \{\tilde h_x > s\}||(M) \, ds,\]
where $||\partial E||$ denotes the total variation measure of ${\bf 1}_E$ for a set $E \subset M$ of locally finite perimeter and $\llbracket U \rrbracket$ denotes the mod 2 flat $n$-current associated to an open set $U \subset M$. The equality on the left follows from \cite[Remark 27.7]{SimonBook}. Furthermore, by (\ref{normalization}) and the coarea formula for BV functions (see \cite[\S 5.5]{EvansGariepy}),
	\[\int_{-\tilde \sigma}^{\tilde \sigma} ||\partial \{\tilde h_x > s\}||(M)\,ds \leq ||D\tilde h_x||(M) = \int_M|\nabla \tilde h_x| \leq E_\e(h_x)/2.\] Which implies \[ \M(\partial \llbracket \{\tilde h_x > \tilde s_x\}\rrbracket) \leq E_\e(h_x)/4\tilde\sigma.\] 

By the symmetry of $h$, we see that $\tilde s_x$ may be chosen so that $\tilde s_{-x} = -\tilde s_x$. Since for every $x\in X$ the set of all $s \in [-\tilde \sigma, \tilde \sigma]$ for which the set $\{\tilde h_x = s\}$ has positive $\hm^n$ measure is at most countable, we can also assume that $\hm^n(\{\tilde h_x = \tilde s_x\})=0$. This implies
	\[\llbracket \{ \tilde h_x > \tilde s_x \} \rrbracket - \llbracket \{\tilde h_{-x} > \tilde s_{-x} \}\rrbracket = \llbracket M \setminus \{\tilde h_x= \tilde s_x\} \rrbracket = \llbracket M \rrbracket, \quad \mbox{for all} \quad x \in X.\]

Since $\Psi$ is odd and strictly increasing, we can choose $\delta \in (0,1)$ depending only on $\tilde \sigma$, so that $s_x :=\Psi^{-1}(\tilde s_x) \in (-1+\delta,1-\delta)$ for all $x \in X$. Now, for $j \in \N$ we define $\phi_0:X(j)_0 \to \znm$ as
	\[\phi_0(x) = \partial \llbracket \{h_x > s_x\} \rrbracket.\] 
This is a good discrete approximation of a $p$-sweepout since, as we saw above, we have a control for its mass in terms of the energies $E_\e(h_x)$. It is left to show that  $\phi_0$ is arbitrarily fine with respect to the flat metric provided we we pick $j \in \N$ sufficiently large. To obtain this control of the fineness in the flat norm we use the following lemma, which is a restatement of \cite[Lemma 8.11]{Guaraco}, to construct auxiliary currents $\Sigma_x$ which vary finely with respect to $x$ and are close to $\phi_0(x)$ in the flat norm.

\begin{lemma} \label{lema:nivel}
Let $\delta \in (0,1)$ and $\alpha \in (-1+\delta,1-\delta)$. For every $h \in \Gamma(X)$, write
	\[\Omega_x = \{h_x >\alpha\}.\]
Given $\e>0$, there exists $\zeta=\zeta(\delta,h,\alpha,\e)>0$ such that
	\[\hm^n(\Omega_x \setminus \Omega_y) \leq 2C_\delta^{-1} \e  \sup_{h(X)} E_\e,\]
for all $x,y \in X$ such that $|x-y|<\zeta$, where $C_\delta=W(1-\delta)>0$.
\end{lemma}


Choose $\alpha \in (-1+\delta, 1-\delta)$ such that, for all $x \in X \cap \Q^m$, the open sets $\Omega_x=\{h_x>\alpha\}$ have finite perimeter and consider the positive number $\zeta=\zeta(\delta,h,\alpha,\e)$ given by the previous lemma. Choosing $j \in \N$ such that $|x-y|<\zeta$ whenever $x$ and $y$ are vertices of $X(j)_0$ which lie in a common $p$-cell of $X(j)$.

For such $x,y \in X(j)_0$, if we write $\Sigma_x = \partial \llbracket \Omega_x \rrbracket$ then
	\[\mathcal{F}(\Sigma_x, \Sigma_y) \leq \M(\llbracket \Omega_x \rrbracket - \llbracket \Omega_y \rrbracket) \leq \hm^n(\Omega_x \setminus \Omega_y) + \hm^n(\Omega_y \setminus \Omega_x).\] In particular, by the lemma above, for $x,y$ in a common $p$-cell we have $\mathcal{F}(\Sigma_x, \Sigma_y) \leq 4C_\delta^{-1} \e  \sup_{h(X)} E_\e$.
	
On the other hand
	\begin{align*}
		\mathcal{F}(\phi_0(x), \Sigma_x) & \leq \M\left(\llbracket \{h_x > s_x\} \rrbracket - \llbracket \Omega_x \rrbracket\right)\\
		& \leq \hm^n\left( \{s_x < h_x \leq \alpha\} \right) + \hm^n\left( \{\alpha < h_x \leq s_x\}  \right)\\
		& \leq \hm^n\left( \{|h_x| \leq 1-\delta\} \right) \leq C_\delta^{-1}\e E_\e(h_x).
	\end{align*}
Where the last inequality follows directly from the definition of $E_\e$ and the hypothesis on $W$ (\cite[Lemma 8.10]{Guaraco}).

Finally, we obtain, for every such pair of vertices,
	\[\mathcal{F}(\phi_0(x),\phi_0(y)) \leq 6C_\delta^{-1}\e \cdot \sup_{x\in X}(E_\e \circ h_x).\]

Given $\tilde\rho>0$, we choose $\e>0$ such that $6\e(c_\e(p)+\e) < \tilde \rho C_\delta$, a complex $X \in \mathcal{C}_p$ and $h \in \Gamma(X)$ such that $\sup(E_\e \circ h) \leq c_\e(p) + \e$ (note that $\delta$ depends only on $\tilde \sigma)$. For this choice of $X$, the map $\phi_0$ satisfies
	\[\mathcal{F}(\phi_0(x),\phi_0(y)) < \tilde\rho,\]
for every pair of vertices $x,y \in X(j)_0$ which lie in a common $p$-cell of $X(j)_0$ (in particular, for adjacent vertices). Furthermore,
	\[\sup_{x \in X(j)_0} \M(\phi_0(x)) \leq \sup_{x \in X(j)_0} \frac{E_\e(h_x)}{4 \tilde \sigma} \leq \frac{c_\e(p)+\e}{4 \tilde \sigma}.\]

\begin{remarq}
The calculations above are essentially the reason why the discrete map $\phi_0$ will give rise to a $p$-sweepout when interpolated. They show that $\phi_0=\partial \Omega$ for a discrete map $\Omega:X(j)_0 \to {\bf I}_n(M,\Z/2)$ with small fineness in the mass norm satisfying $\Omega(x)+\Omega(-x) = \llbracket M \rrbracket$ for all $x \in X(j)_0$. It follows that the image of every closed discrete path in $X(j)_0$ under $\phi_0$ is the image of a closed discrete path in ${\bf I}_n(M,\Z/2)$ under the boundary operator and hence it must be in the kernel of Almgren's isomorphism. Similarly, a discrete path $\alpha$ in $X(j)_0$ joining a pair of antipodal points is mapped by $\phi_0$ to the boundary of a discrete path in ${\bf I}_n(M,\Z/2)$ with the property that the sum of its extremes equals $\llbracket M \rrbracket$; this shows that $F_A([\alpha])$ is non-trivial. See the following Remark \ref{loop} for more details.
\end{remarq}

\subsection{Construction of a $p$-sweepout: Interpolation from the flat to the mass norm}

The discrete map $\phi_0:X(j)_0 \to \znm$ constructed in the last subsection can be interpolated to produce a new discrete map with small fineness and such that its Almgren extension induces a $p$-sweepout defined on the orbit space $\tilde X$. We will now describe this interpolation procedure.

The fundamental result we will need is the following theorem, which is a consequence of \cite[Proposition 5.8]{Zhou2015} and the compactness, with respect to the flat topology, of the space of sets of finite perimeter and perimeter bounded above by a constant $L>0$.

\begin{theorem}\label{itp}
Let $\delta,L>0$. There exist $\eta = \eta(\delta, L) \in (0,\delta)$, $\ell = \ell(\delta, L) \in \N$, and a function $\xi=\xi_{(\delta,L)}:(0,+\infty) \to (0,+\infty)$ such that $\xi(s) \to 0$ as $s \to 0^+$ with the following property. Given $i,j_0 \in \N$ with $i \leq p$, $s \in (0,\eta)$ and a discrete map \(\phi:I_0(i,j_0)_0 \to \znm\) such that:
	\begin{enumerate} 
		\item $\mathcal{F}(\phi(x), \phi(y)) < s$ for all $x,y \in I_0(i,j_0)_0$;
		\item $\sup_{x \in I_0(i,j_0)_0}\M(\phi(x))\leq L$;
		\item for each $x \in I_0(i,j_0)_0$, there exists a set $U_x \subset M$ of finite perimeter such that $\phi(x) = \partial\llbracket U_x \rrbracket$,
	\end{enumerate}
there exists $\tilde \phi: I(i,j_0+\ell)_0 \to \znm$ satisfying:
	\begin{enumerate} {\renewcommand{\theenumi}{\roman{enumi}} \renewcommand{\labelenumi}{(\textit{\theenumi})}}
		\item $\mathcal{F}(\tilde \phi(x), \tilde \phi(y))<\xi(s)$ for all $x,y \in I(i,j_0+\ell)_0$;
		\item $\sup_{x \in I(i,j_0+\ell)_0}\M(\tilde \phi(x)) \leq \sup_{I_0(i,j_0)_0}\M(\phi) + \delta$
		\item $\tilde \phi = \phi \circ {\bf n}(j_0+\ell,j_0)$ on $I_0(i,j_0+\ell)_0$;
		\item for each $x \in I(i,j_0+\ell)_0$, there exists a set $V_x \subset M$ of finite perimeter such that $\tilde \phi(x) = \partial\llbracket V_x \rrbracket$, and $V_x=U_x$ if $x \in I_0(i,j_0)_0$.
		\item ${\bf f}(\tilde \phi) \leq \delta$, if $i=1$, and ${\bf f}(\tilde \phi)\leq b({\bf f}(\phi) +\delta)$ if $i>1$, where $b=b(p)$ is a positive constant;
		\item if $i=1$ and $\delta < \nu_M$ then the sum of the isoperimetric choices for $(\tilde\phi((v+1)3^{-(j_0+\ell)}) - \tilde\phi(v3^{-(j_0+\ell)}))$, for $v=0,\ldots, 3^{j_0+\ell}-1$, equals $T=\llbracket U_{[1]} \rrbracket - \llbracket U_{[0]} \rrbracket$, provided ${\bf M}(T)<\mathrm{Vol}(M)/2. $
	\end{enumerate}
\end{theorem}

We will apply this result inductively to the cells of $X(j)$ to obtain, for each small $\delta_1>0$, an $\e_0=\e_0(\delta_1)\in (0,\delta_1)$ satisfying the conclusion of Theorem \ref{thm:comp}. In particular, for every $\e \in (0,\e_0)$ and every $h \in \Gamma(X)$ such that ${\sup_X(E_\e\circ h) \leq c_\e(p)+\e}$, we will be able to find a discrete map ${\phi_p:X(j_p)_0 \to \znm}$, $j_p \geq j$, with the following properties:
	\begin{enumerate} {\renewcommand{\theenumi}{\roman{enumi}} \renewcommand{\labelenumi}{(\textit{\theenumi})}}
		\item $\M(\phi_p(x)) \leq \frac{1}{4\tilde\delta}(c_\e(p)+C_0\delta_1)$ for all $x \in X(j_p)_0$, for a constant $C_0>0$ depending only on $p$ and $M$;
		\item $\phi_p = \phi_0 \circ {\bf n}(j_p,j)$ on $X(j)_0$;
		\item ${\bf f}(\tilde \phi_p)\leq c \cdot \delta_1$, for a constant $c=c(p)>0$;
		\item For each $1$-cell $\tau \in X(j)_1$, if $\alpha_\tau:[0,1] \to \tau$ is an affine homeomorphism, then
			\[Q_{\tau} = \llbracket \{h_{\alpha_\tau(1)}>s_{\alpha_\tau(1)}\} \rrbracket - \llbracket \{h_{\alpha_\tau(0)}>s_{\alpha_\tau(0)}\} \rrbracket,\]
			where $Q_\tau$ is the sum of the isoperimetric choices for the currents ${(\phi_p\circ \alpha_\tau)((v+1)3^{-\ell_p})-(\phi_p\circ \alpha_\tau)(v3^{-\ell_p})}$ for $v=0,\ldots, 3^{\ell_p}-1$ and ${\ell_p=j_p-j}$.
	\end{enumerate}
	
\begin{remarq} \label{loop}
The last item above will guarantee that if $\tau_1,\ldots, \tau_a \in X(j)_1$ are such that we can choose the corresponding homeomorphisms $\alpha_{\tau_v}$ in a way that $\alpha_{\tau_v}(1)=\alpha_{\tau_{v+1}}(0)$, for $v=1,\ldots, a-1$, then
	\begin{align*}
		\sum_{v=1}^a Q_{\tau_v} & = \sum_{v=1}^a \left(\llbracket \{h_{\alpha_{\tau_v}(1)}>s_{\alpha_{\tau_v}(1)}\} \rrbracket - \llbracket \{h_{\alpha_{\tau_v}(0)}>s_{\alpha_{\tau_v}(0)}\} \rrbracket \right) \\
		& = \llbracket \{h_{\alpha_{\tau_a}(1)}>s_{\alpha_{\tau_a}(1)}\} \rrbracket - \llbracket \{h_{\alpha_{\tau_1}(0)}>s_{\alpha_{\tau_1}(0)}\}\rrbracket.
	\end{align*}
In particular, if $\alpha_{\tau_a}(1)=\alpha_{\tau_1}(0)$, we have $\sum_{v=1}^a Q_{\tau_v} = 0$, while $\alpha_{\tau_a}(1) = - \alpha_{\tau_1}(0)$ implies
	\[\sum_{v=1}^a Q_{\tau_v} = \llbracket \{h_{\alpha_{\tau_a}(1)}>s_{\alpha_{\tau_a}(1)}\} \rrbracket - \llbracket \{h_{\alpha_{\tau_1}(0)}>s_{\alpha_{\tau_1}(0)}\} \rrbracket = \llbracket M \rrbracket.\]
These facts will be used to verify the topological non-triviality of the Almgren extension of $\phi_p$.
\end{remarq}

Let $\delta_1\in (0,\min\{\nu_M, c(p,M)\})$, where $c(p,M)$ is a constant depending only on $p$ and $M$ to be chosen later, and
	\[L_1 = \frac{\sup_{\e \in (0,1]} c_\e(p)+1}{4\tilde\delta}.\]
For each $i=1,\ldots, p$, let $L_i = L_1+(i-1)\delta_1$, $\ell_i=\sum_{v=1}^p\ell(\delta_1,L_v)$ and $\xi_i = \xi_{(\delta_1,L_i)}$. We choose $s_1 \in (0,\eta(\delta_1,L_1)/4)$ such that $s_1 < \mathrm{Vol}(M)/2$ and
	\[2(i+1)(\xi_i \circ \xi_{i-1} \circ \ldots \circ \xi_1)(s_1) < \eta(\delta_1,L_{i+1}), \ \mbox{for} \ i=1,\ldots, p-1.\]
There exists $\e_0=\e_0(\delta_1) \in (0,\delta_1)$ such that for all $\e \in (0,\e_0)$ we have ${6\e(c_\e(p)+\e) < s_1 C_\delta}$. For such $\e$, choose $h \in \Gamma(X)$ for some $X \in \mathcal{C}_p$ such that $\sup(E_\e\circ h)<c_\e(p) + \e$. Then the construction of the last subsection gives a map $\phi_0:X(j)_0 \to \znm$ satisfying
	\begin{align*}
		\mathcal{F}(\phi_0(x),\phi_0(y)) & \leq {\bf M}\left(\llbracket \{h_x > s_x\} \rrbracket - \llbracket \{h_y > s_y\} \rrbracket\right) \\
		& \leq \frac{6\e(c_\e(p)+\e)}{C_\delta}<s_1< \eta(\delta_1,L_1)
	\end{align*}
for every pair of vertices $x,y \in X(j)_0$ which lie in a $p$-cell of $X(j)_p$, and
	\[\sup_{x \in X(j)_0} \M(\phi_0(x)) \leq \frac{c_\e(p)+\e}{4\tilde \sigma} \leq L_1.\]
For each $i=1,\ldots, p$, denote by $V_i$ the set of vertices of $X(j+\ell_i)$ which lie in the $i$ skeleton of $X(j)$, i.e. 
	\[V_i = \bigcup_{\tau \in X(j)_i} \tau(\ell_i)_0 = \bigcup_{\gamma \in X(j)_{i+1}}(\gamma_0(\ell_i))_0.\]
Note that $V_p = X(j+\ell_p)_0$, since $X$ has dimension $p$. Applying Theorem \ref{itp} to each $1$-cell of $X(j)$, we get a map
	\[\phi_1: V_1\to \znm,\]
such that, for each $\tau \in X(j)_1$,
	\[\mathcal{F}(\phi_1(x), \phi_1(y)) \leq \xi_1(s_1) < \eta(\delta_1,L_2)\]
for $x,y \in \tau(\ell_1)_0$, and
	\[\sup_{x \in V_1}\M(\phi_1(x)) \leq \sup_{x \in X(j)_0} \M(\phi_0(x)) +\delta_1 \leq \frac{c_\e(p)+\e}{4\tilde \sigma} + \delta_1 \leq L_2.\]
Moreover, $\phi_1$ is given by the boundary currents induced by subsets of $M$ of finite perimeter, agrees with $\phi_0$ on the domain of $\phi_0$ and satisfies \({\bf f}(\phi_1) \leq \delta_1\). Finally, since $\phi_0$ is even, we may assume that the map $\phi_1$ is also even, after possibly redefining $\phi_1$ in half of the cells of $X(j)_1$. 

Inductively, given an even map
	\[\phi_{i-1} : V_{i-1} \to \znm\]
satisfying:
	\begin{enumerate}
		\item $\mathcal{F}(\phi_{i-1}(x),\phi_{i-1}(y)) \leq (\xi_{i-1} \circ \ldots \xi_1)(s_1) < \eta(\delta_1,L_{i})$ for every pair of vertices $x,y \in \tau(\ell_{i-1})_0$ for some $\tau \in X(j)_{i-1}$,
		\item $\sup_{x \in V_{i-1}}\M(\phi_{i-1}(x)) \leq (c_\e(p)+\e)/{4\tilde \sigma}+(i-1)\delta_1\leq L_i$,
		\item for every $x \in V_{i-1}$, we have $\phi_{i-1}(x) = \partial \llbracket E_x \rrbracket$ for some $E_x \subset M$ of finite perimeter, and
		\item ${\bf f}(\phi_{i-1}) \leq b_{i-1} \delta_1$, for some $b_{i-1}=b_{i-1}(p)>0$,
	\end{enumerate}
we apply Theorem \ref{itp} on each $i$-cell of $X(j)$ to obtain a new even map
	\[\phi_i: V_i \to \znm\]
such that:
	\begin{enumerate} {\renewcommand{\theenumi}{\roman{enumi}} \renewcommand{\labelenumi}{(\textit{\theenumi})}}
		\item $\mathcal{F}(\phi_i(x),\phi_i(y)) < (\xi_i \circ \ldots \xi_1)(s_1)$, for all $x,y \in \tau(\ell_i)_0$, 
		\item $\sup_{x \in V_i}\M(\phi_i(x)) \leq \sup_{x \in V_{i-1}}\M(\phi_{i-1}(x))+\delta_1\leq (c_\e(p)+\e)/{4\tilde \sigma}+i\delta_1\leq L_{i+1} $;
		\item $\phi_i=\phi_{i-1} \circ {\bf n}(j+\ell_i,j+\ell_{i-1})$ on the vertices of the $(i-1)$ skeleton of $V_i$;
		\item $\phi_i(x)$ is the boundary of a chain induced by a set of finite perimeter, for all $x \in V_i$; and
		\item ${\bf f}(\phi_i) \leq b({\bf f}(\phi_{i-1})+\delta_1) \leq b_i \delta_1$, for $b_i = b(b_{i-1}+1)$.
	\end{enumerate}
For $i=p$, we obtain a discrete map ${\phi_i:V_p=X(j+\ell_p)_0 \to \znm}$ with fineness ${{\bf f}(\phi_p) \leq b_p\delta_1 < b_p c(p,M)}$.
If $c(p,M)>0$ is sufficiently small, then we can apply Theorem \ref{Alm_ext} to obtain the Almgren extension $\Phi:X \to \znm$ of $\phi_p$. Since
	\[\M(\Phi(x)-\Phi(y)) \leq C_0(p,M){\bf f}(\phi_p), \]
whenever $x, y \in X$ lie in a common $p$-cell of $X(j+\ell_p)$, we get
	\begin{align*}\sup_{x \in X}\M(\Phi(x)) & \leq \sup_{x \in V_p}\M(\phi_p(x)) + C_0(p,M) b_p \delta_1 \\
	& \leq \frac{c_\e(p)+\e}{4\tilde \sigma}+(p + C_0(p,M)b_p) \delta_1.
	\end{align*}
From the fact that for each cell $\alpha$ of $X(j+\ell_p)$ the map $\Phi|_\alpha$ depends only on the values of $\phi(x)$ for $x \in \alpha_0$ it follows that $\Phi$ may be assumed to be an odd map.

\subsection{Construction of a $p$-sweepout: non-triviality} In order to conclude the proof of Theorem \ref{thm:comp}, it remains to show that the map $\Phi$ induces a $p$-sweepout ${\tilde\Phi:\tilde X \to \znm}$. By Proposition \ref{prop:pswo}, it suffices to show that $\ker \tilde\Phi_* = p_*\pi_1(X)$, where $p:X \to \tilde X$ is the orbit map. This is accomplished by the following observations. Let $\alpha:[0,1] \to X$ be a path such that, for some $v \in \N$, the restriction of $\alpha$ to each interval $[t_i,t_{i+1}]$ for $i=0,\ldots, 3^v-1$ is an affine homeomorphism onto a $1$-cell $\tau_i \in X(j)_1$, where $t_i=i3^{-v}$. Fix $i$ and let $t_{i,j} = t_i+j3^{-(v+\ell_p)}$ for $j=0,\ldots, 3^{\ell_p}$. It follows from the Remark \ref{loop} that if $Q_{i,j}$ is the isoperimetric choice for
	\[\Phi(\alpha(t_{i,j+1})) - \Phi(\alpha(t_{i,j})) = (\phi_p \circ \alpha)(t_{i,j+1})-(\phi_p \circ \alpha)(t_{i,j})\]
then
	\begin{align*}
		\sum_{i=0}^{3^v-1}\sum_{j=0}^{3^{\ell_p}-1} Q_{i,j} & = \sum_{i=0}^{3^v-1} \left( \llbracket \{ h_{\alpha(t_{i+1})} > s_{\alpha(t_{i+1})} \} \rrbracket - \llbracket \{ h_{\alpha(t_{i+1})} > s_{\alpha(t_{i+1})} \} \rrbracket\right) \\
		& = \llbracket \{h_{\alpha(1)} > s_{\alpha(1)} \} \rrbracket - \llbracket \{h_{\alpha(0)} > s_{\alpha(0)} \} \rrbracket.
	\end{align*}
If $\alpha$ is a closed path, then
	\[F_A([\Phi \circ \alpha]) = \left[ \sum_{i=0}^{3^v-1}\sum_{j=0}^{3^{\ell_p}-1} Q_{i,j} \right]=0.\]
Noting that $F_A$ is an isomorphism and $[\Phi \circ \alpha] = [\tilde\Phi \circ p \circ \alpha]$ we get
	\[\tilde\Phi_*\left( p_*[\alpha] \right) = 0, \quad \mbox{in} \quad \pi_1(\znm).\]
This proves that $\ker \tilde\Phi_* \supset p_*(\pi_1(X))$. On the other hand, if $\alpha$ satisfies ${\alpha(1)=-\alpha(0)}$, then
	\[F_A([\Phi \circ \alpha]) = \left[ \sum_{i=0}^{3^v-1}\sum_{j=0}^{3^{\ell_m}-1} Q_{i,j} \right]=\left[ M \right ]\]
and $[p \circ \alpha] \not\in \ker \tilde \Phi_*$. Since the lift of every closed path $\tilde \alpha$ in $\tilde X$ with $\tilde \alpha(0)$ in $p(X(j)_0)$ is homotopic (with fixed endpoints) to such a poligonal path which is either closed or either joins antipodal points, this proves that $\ker \tilde\Phi = p_*(\pi_1(X))$. Therefore, we conclude that $\tilde\Phi$ is a $p$-sweepout. This completes the proof of Theorem \ref{thm:comp}.

\subsection{Min-max values in \texorpdfstring{$S^n$}{S^n}}

We can use the comparison between $\omega_p(M)$ and the min-max values for the energy to calculate the value of $\lim c_\e(p)$, when $\e \downarrow 0$, for small $p$ and $M=S^n$. More precisely, we will prove that
	\begin{equation} \label{ineq_sn}
		\frac{1}{2\sigma} \limsup_{\e \to 0^+} c_\e(p) \leq \mathrm{Vol}(S^{n-1}),
	\end{equation}
for $p=1,\ldots, n+1$. Since $S^{n-1} \subset S^n$ is the minimal surface of least area in $S^n$, by \cite[Theorem 2.14 and Lemma 4.7]{MarquesNevesInfinite} we have $\omega_p(S^n)\geq\mathrm{Vol}(S^{n-1})$. Then, it follows from Theorem \ref{thm:wplp} and \eqref{ineq_sn} that
	\[\frac{1}{2 \sigma} \lim_{\e \to 0^+} c_\e(p) = \omega_p(S^n) = \mathrm{Vol}(S^{n-1}), \quad \mbox{for} \quad p=1,\ldots, n+1.\]
We will construct an odd continuous map $h:S^{p} \to H^1(M)$ with
	\[\sup_{a \in S^p} E_\e(h(a)) \leq 2 \sigma \mathrm{Vol}(S^{n-1}).\]
For every $a=(a_0,\ldots, a_{n+1}) \in S^{n+1}$, define $f_a:S^n \to \R$ by
	\[f(x) = a_0 + \sum_{i=1}^{n+1} a_ix_{i-1}, \quad \mbox{for} \quad x=(x_0,\ldots, x_n) \in S^n, \]
and let $S_a = \{f_a=0\} \subset S^n$. Denote by $d_a$ the signed distance function to $S_a$, that is,
	\[d_a(x) = \left(\mathrm{sgn} f_a(x)\right)\cdot d(x,S_a), \quad \mbox{for} \quad x \in S^n.\]
Observe that $S_{-a}=S_a$ and $d_{-a}=-d_a$ for all $a \in S^{n+1}$. We define ${v_{a,\e}:S^{n+1} \to \R}$ by
	\[v_{a,\e}(x) = \psi(d_a(x)/\e),\]
where $\psi:\R \to \R$ is the unique monotone solution to the Allen-Cahn equation with $\e=1$ and $\psi(0)=0$. Since $\psi$ is odd, we can define an odd continuous map $h:S^{n+1} \to H^1(S^n)$ by $h(a)=v_{a,\e}$. Thus, $h|_{S^p} \in \Gamma(S^p)$ for all $p=1,\ldots, n+1$. Note also that both $S_a$ and the level sets for $d_a$ are geodesic spheres in $S^n$. Thus,
	\[\sup_{s \in \R} \hm^{n-1}(\{d_a = s\}) \leq \mathrm{Vol}(S^{n-1}), \quad \mbox{for all} \quad a \in S^{n+1}.\]
We can proceed as in \cite[\S 7]{Guaraco} to shows that
	\[E_\e(h(a)) \leq \int_{-\infty}^{+\infty}\left[ \psi'(s)^2 +W(\psi(s)) \right] \hm^{n-1}(\{d_a=\e s\})\,ds \leq 2\sigma \cdot \mathrm{Vol}(S^{n-1})\]
and consequently
	\[c_\e(p)\leq c_\e(S^p) \leq \sup_{a \in S^{p}} E_\e(h(a)) \leq 2\sigma \cdot \mathrm{Vol}(S^{n-1}), \quad \mbox{for} \quad p=1,\ldots, n+1.\]
This proves the inequality \eqref{ineq_sn}.


\begin{appendix}

\section{Appendix: Technical Lemmas}

We list here three technical lemmas. We omit the proofs of the first two since they follow from elementary properties of Sobolev spaces and Lipschitz functions.

\begin{lemma} \label{truncar} Let $U$ be a bounded open set of a manifold. The function $\max (\cdot,\cdot) : H^1(U)\times H^1(U) \to  H^1(U)$ is continuous. The same holds for $\min(\cdot,\cdot)$.\end{lemma}

\begin{lemma}\label{lipschitz}
Let $G:K\to M$ be a bi-Lipschitz map from a cubical complex $K$ to a compact manifold $M$ and let $h:K\to \R$ be a Lipschitz function.

\begin{enumerate}
\item There exists a constant $C>0$ such that $$C^{-1}\sum_{\sigma \in K(j)_n} \| h \|_{H^1(\sigma)} \leq \| h\circ G^{-1} \|_{H^1(M)} \leq C\sum_{\sigma \in K(j)_n} \| h \|_{H^1(\sigma)}.$$
\item Let $h_k: K \to \R$, for $k\in\N$, be a sequence of Lipschitz functions with bounded Lipschitz constants. If $ h_k \to h$ and $\nabla h_k \to \nabla h$ a.e. on $K$ then $h_k \circ G^{-1} \to h \circ G^{-1}$ in $H^{1}(M)$.
\item Let $h_k: K \to \R$, for $k\in\N$, be a sequence of Lipschitz functions. If $ h_k \to h$ and $\nabla h_k \to \nabla h$ a.e. on $K$ then $ h^{+}_k \to h^{+}$ and $\nabla h^{+}_k \to \nabla h^{+}$ a.e. on $K$.

\end{enumerate}

\begin{lemma}\label{raizeshausdorff} Let $a_n \to a$ be a convergent sequence in $S^p$. Assume that $C(a)$ is not empty and let $D\in \mathbb{C}$ an open disk containing $C(a)$. Let $\xi=(z_1,\dots,z_j)$ be an array of the roots of $p_a$ repeated according to its multiplicities. Then, for any $n$ big enough, there is an array $\xi_n=(z_1(n),\dots,z_j(n))$ of the roots of $p_{a_n}$ contained in $D$, repeated according to its multiplicities, and such that $\xi_n\to \xi$. If $C(a)$ is empty, i.e. $a=(\pm 1,0,\dots,0)$, then $C(a_n)$ is either empty, or goes to infinity uniformly.\end{lemma}

\begin{proof}
Notice that $p_{a_n} \to p_a$ uniformly in compact sets. Let $\tilde D$ be any disk with $\partial \tilde D\cap C(a)=\emptyset$. Then, for $n$ big enough, we have $|p_{a_n}-p_a|<|p_a|$ on $\partial \tilde D$ and by Rouch\'e's Theorem $p_a$ and $p_{a_n}$ have the same number of roots in $\tilde D$ counted with multiplicities. This already proves the lemma when $a=(\pm 1,0,\dots,0)$.

Let $D$ be as in the statement of the lemma. By the argument above, for $n$ big enough there are exactly $j$ roots of $p_{a_n}$ in $D$, counted with multiplicities. Define $D_i(\delta)=\{z: |z-z_i|<\delta\}$ for some $\delta>0$ small enough so that $D_i(\delta)\cap D_j(\delta)=\emptyset$ if $z_i\neq z_j$ and $D_i(\delta) \subset D$. Let $m_i$ be the multiplicity of $z_i$ as a root of $p_a$. As before, for $n$ big enough, the number of roots of $p_{a_n}$ in $D_i(\delta)$, counted with multiplicities, is $m_i$. Then, for this fixed $\delta$ and every $i$,  define $\xi_n(i) = (z_1(i),\dots,z_{m_i}(i))$ as any array of these $m_i$ roots, repeated according to its multiplicity. The set $C(a_n)\cap D_i(\delta) \to \{z_i\}$ in the Hausdorff distance, since for every $\e>0$ we have that the number of roots with multiplicities in $D_i(\e)$ is also $m_i$ if $n$ is big enough. This implies $\xi_n(i) \to (z_i,\dots,z_i) \in \mathbb{C}^{m_i}$. We can choose $n$ big enough so that the arguments above carry over all roots $z_i$ of $p_a$. Now we define $\xi_n$ as the juxtaposition of the $\xi_n(i)$. This $\xi_n$ satisfies the conclusions of the lemma.
\end{proof}

\end{lemma}

\section{A cohomological index theory for free \texorpdfstring{$\Z_2$}{Z_2} actions} \label{ap:index}
In this Appendix we fill in some details about the topological $\Z/2$-index of Fadell and Rabinowitz \cite{Fadell-Rabinowitz1}. We follow the general description given in \cite{Fadell-Rabinowitz2}, which works for actions of any compact Lie group, restricting ourselves to the case of $\Z/2$ actions.

Let $S^\infty$ be the infinite dimensional unit sphere, that is, the direct limit of the family of topological spaces $\{S^n\}_{n \in \N\cup\{0\}}$, directed by the inclusions $S^n \hookrightarrow S^m, n\leq m$. This means that $S^\infty=\bigcup_{n \geq 0}S^n$ and a map $f:S^\infty \to Z$ to an arbitrary topological space $Z$ is continuous if, and only if, $f \circ \iota_n:S^n \to Z$ is continuous for all $n$, where $\iota_n:S^n \hookrightarrow S^\infty$ is the inclusion. Similarly, we consider the infinite dimensional real projective space $\R\p^\infty = \bigcup_n \R\p^n$, which can be seen as the orbit space of $S^\infty$ by the free $\Z/2$ action given by the antipodal map $x \in S^n \mapsto -x \in S^n$. We denote by $\mathrm{pr}:S^\infty \to \R\p^\infty$ the orbit map. 

The cohomology ring of $\R\p^\infty$ with $\Z/2$ coefficients is isomorphic to 	the polynomial ring $\Z/2[w]$ over $\Z/2$, where ${w \in H^1(\R\p^\infty;\Z/2)}$ (see \cite[Theorem 3.12]{Hatcher}). In the following, we will use the Alexander-Spanier cohomology with $\Z/2$ coefficients and refer to \cite{Spanier} for the definition and basic properties of this cohomology theory. We remark that it agrees with the singular cohomology for any CW complex, hence for $\R\p^\infty$ (see \cite[\S 6.9]{Spanier} for a more general statement).

Lets recall the definition of the cohomological index. Denote by $\mathcal{F}$ the class of all $\Z/2$-spaces. For every free $(X,T) \in \mathcal{F}$ we can find a continuous equivariant map $f:X \to S^\infty$, and then a continuous map $\tilde f:\tilde X = X/\{x\sim Tx\} \to \R\p^\infty$ such that $\mathrm{pr} \circ \tilde f = f \circ p$, where $p:X \to \tilde X$ is the orbit map. This map is called a \emph{classifying map} for the $\Z/2$ action given by $T$. It is possible to show that $\tilde f$ is unique up to homotopy, and therefore we have an induced map $\tilde f^* : H^*(\R\p^\infty;\Z/2) \to H^*(\tilde X;\Z/2)$ which depends only on $(X,T)$. We define the \emph{cohomological index} of $(X,T)$ by
	\[ \ind_{\Z/2}(X,T) = \ind_{\Z/2}(X) = \sup\{k \in \N : \tilde f^*(w^{k-1}) \neq 0 \in H^{k-1}(\tilde X;\Z/2) \}.\]
We set $w^0 = 1 \in H^0(\R\p^\infty;\Z/2)$ and adopt the convention $\ind_{\Z/2}(\emptyset,T) = 0$. We will omit $T$ in the notation above whenever the action is understood. If $(X,T) \in \mathcal{F}$ is not free, we set $\ind_{\Z/2}(X,T) = \infty$.

\begin{remarq} This definition of $\ind_{\Z/2}(X,T)$ for non-free $(X,T) \in \mathcal{F}$ agrees with the one given in \cite{Fadell-Rabinowitz2} since $(X,T) \in \mathcal{F}$ is free if and only if there are no trivial orbits. More precisely, if $T\neq \mathrm{id}_X$ and $\bar x \in X$ is such that $T(\bar x)=\bar x$ then $S^\infty \times \{x\}$ with the diagonal action $(y,x) \mapsto (-y,T(x))$  is isomorphic to $S^\infty$ as $\Z/2$-spaces. Since $\ind_{\Z/2}(S^\infty) = \infty$ the monotonicity property (see \ref{thm:indice} below) implies $\ind_{\Z/2}(S^\infty \times X) =\infty$. We point out that this fact does not hold for actions of other compact Lie groups, see \cite[\S 6]{Fadell-Rabinowitz2}.
\end{remarq}

\vspace{5pt}

The existence and uniqueness of the classifying map can be seen as a direct consequence of the classification of $G$-principal bundles (see \cite{Dold}), since $\R\p^\infty$ is the classifying space for $\Z/2$. We can give an explicit construction of this map when $X$ is compact, following Milnor's construction of the classifying bundle in \cite{Milnor2}. We embed $S^\infty$ in a separable Hilbert space $H$ by choosing a countable orthonormal set $\{e_i\}_{i \in \N}$ and identifying $S^n$ with the unit sphere in the subspace generated by $\{e_1,\ldots, e_{n+1}\}$. Choose a finite open cover $\{\tilde U_i\}_{i=1}^N$ for $\tilde X$ such that each $\pi^{-1}(\tilde U_i)$ is the disjoint union $U_i \cup T(U_i)$ for an open set $U_i \subset X$ (note that the orbit map $\pi:X \to \tilde X$ is a covering map). We can also choose a partition of the unity $\{\rho_i\}_{i =1}^N\cup \{\zeta_i\}_{i=1}^N$ subordinated to the open cover $\{U_i\}_{i=1}^N \cup \{T(U_i)\}_{i=1}^N$ such that $\zeta_i=\rho_i \circ T$. Hence, there is no $x \in X$ such that $\rho_i(x)=\rho_i(Tx)$ for all $i=1,\ldots, N$ and we can define a continuous map $f:X \to S^\infty$ by:
	\[f(x) = \sum_{i=1}^N \frac{\rho_i(x) - \rho_i(Tx)}{\rho(x)}e_i, \quad \mbox{where} \quad \rho(x) = \left(\sum_{i=1}^N (\rho_i(x)-\rho_i(Tx))^2\right)^{1/2}.\]
Since this map satisfies $f\circ T=-f$, it induces a continuous map ${\tilde f: \tilde X \to \R\p^\infty}$ such that $p \circ f = \tilde f \circ \pi$, and gives the classifying map for $(X,T)$. Using partitions of unity for paracompact spaces, this construction can be adapted for a general $(X,T) \in \mathcal{F}$. For an elementary proof of the uniqueness of $f$ modulo homotopy, see \cite[\S 14.4]{tomDieck}. 

The fact that $\ind_{\Z/2}$ is a topological $\Z/2$-index is a consequence of the following theorem.

\begin{theorem} \cite{Fadell-Rabinowitz2} \label{thm:indice} The cohomological index is a topological $\Z/2$-index, that is, it satisfies the properties (I1)-(I6) of Definition \ref{def:indice} and the following stronger version of $(I5)$: 
	\begin{enumerate}
		\item[(I5)'] For every free $\Z/2$-space $(X,T) \in \mathcal{F}$, the orbit space $\tilde X$ is infinite whenever $\ind_{\Z/2}(X,T)\geq 2$.
	\end{enumerate}
Moreover, we have the following \emph{dimension property}: if $(X,T) \in \mathcal{F}$ and $\dim X$ denotes the covering dimension of $X$, then $\ind_{\Z/2}(X,T) \leq \dim X$.
\end{theorem}

We point out that the continuity property (I3) is a consequence of the tautness of the Alexander-Spanier cohomology \cite[\S 6.1]{Spanier} -- the monotonicity and the fact that every invariant closed subset $A \subset X$ is contained in a invariant paracompact neighborhood $N \subset X$. Moreover, the subaditivity (I4) is a restatement of the vanishing property of the cup product \cite[p. 209]{Hatcher} (compare with \cite{Guth}, \cite{Aiex1}). For a complete proof, see \cite[\S 3]{Fadell-Rabinowitz2}.

We conclude our discussion about the Fadell-Rabinowitz topological index $\Z/2$ comparing the cohomological index with another well known $\Z/2$-topological index, previously considered by Yang \cite{Yang2} (where it is referred, up to normalization, as the \emph{$B$-index}), Conner-Floyd \cite{conner-floyd} (where it is called the \emph{co-index}) and Krasnoselskii \cite{Krasnoselskii}. Given a paracompact $\Z/2$ space $(X,T)$, we define
	\[\gamma_{\Z/2}(X,T) = \inf\{ k \in \N : \exists f \in C(X,S^{k-1}) \ \mbox{s.t.} \ f\circ T = -f\}.\]
We also adopt the convention $\gamma_{\Z/2}(\emptyset)=0$. The function $\gamma_{\Z/2}$ defines a $\Z/2$-topological index in the class of all paracompact $\Z/2$ spaces (see e.g \cite[\S 7.3]{Ghoussoub}). It turns out that this index is the maximal topological $\Z/2$-index, as proven in the following: 

\begin{lemma}
For every topological $\Z/2$-index $\ind: \mathcal{C} \to \N \cup\{0,+\infty\}$, it holds
	\[\ind(X,T) \leq \gamma_{\Z/2}(X,T).\]
for all $(X,T) \in \mathcal{C}$.
\end{lemma}

\begin{proof}
Let $k=\gamma_{\Z/2}(X)<\infty$ (if $k=\infty$, there is nothing to prove). Then, we can find a continuous equivariant map $f:X \to S^{k-1}$. For each $i=1,\ldots, k$, let
	\[ A_i = \{x \in S^{k-1} : x_i \neq 0\}, \quad \mbox{and} \quad B_i = f^{-1}(A_i). \]
Since we can find equivariant maps $A_i \to S^0$ and $\ind(S^0) = 1$ (by Definition \ref{def:indice} (I6)), it follows from the monotonicity property that
	\[\ind(B_i) \leq \ind(A_i) \leq \ind(S^0) = 1, \quad \mbox{for all} \quad i=1,\ldots, k.\]
Therefore, from the subaditivity of $\ind$, we get
	\[\ind(X) = \ind\left(\bigcup_{i=1}^kB_i\right) \leq \sum_{i=1}^k \ind(B_i) \leq k = \gamma_{\Z/2}(X).\]
\end{proof}

\begin{remark}
The index $\gamma_{\Z/2}$ has the following important property: for every compact symmetric $X\subset E \setminus\{0\}$, where $E$ is a Banach space, it holds holds $\gamma_{\Z/2}(X) = \mathrm{cat}(\tilde X)$, where $\mathrm{cat}(\tilde X)$ is the \emph{Lusternik-Schnirelmann category} of $\tilde X$, that is, $\mathrm{cat}(\tilde X)$ is the least $k \in \N$ such that $\tilde X$ can be covered by $k$ closed sets $A_i \subset \tilde X$ which are contractible to a point in $\tilde X$. See \cite[Theorem 3.7]{Rabinowitz1} for a proof of this fact, and \cite{Ghoussoub} for some results about the multiplicity of the set of critical points of a functional involving the Lusternik-Schnirelmann category.
\end{remark}


\end{appendix}

\bibliographystyle{siam}
\bibliography{references}

\end{document}